\documentclass[12pt]{article}
\usepackage{geometry}                
\usepackage{graphicx, psfrag, tikz, pgfplots}
\usepackage{subfigure}
\usepackage{amsthm, amsmath, amssymb}
\usepackage{epstopdf, color, bm, enumerate}
\usepackage{authblk, tipa, multirow}
\usetikzlibrary{arrows,intersections}
\usepackage{cancel}
\usepackage{subfig} 

\numberwithin{equation}{section}
\theoremstyle{theorem}
\newtheorem{theorem}{Theorem}[section]
\newtheorem{proposition}[theorem]{Proposition}
\newtheorem{lemma}[theorem]{Lemma}
\theoremstyle{remark}
\newtheorem{remark}[theorem]{Remark}
\theoremstyle{definition}

\title{Global Stability of the Periodic Solution of the Three Level Food Chain Model with Extinction of Top Predator}
\author[1]{Kaijen Cheng
}
\author[2, 3, 4]{Hongming You
}
\author[3]{Ting-Hui Yang\footnote{email : thyang@gms.tku.edu.tw, Research was partially supported by National Council of Science, Taiwan, Republic of China.}}
\affil[1]{School of Mathematics and Statistics, Zhaoqing University, Zhaoqing, Guangdong 526061, PR China.}
\affil[2]{College of Mathematics and Computer Science, Quanzhou Normal University, Quanzhou, 362000, PR China}
\affil[3]{Department of Mathematics, Tamkang University, No. 151 Yingzhuan Road, Tamsui Dist., New Taipei City 25137, Taiwan. (R.O.C.)}
\affil[4]{Fujian Provincial Key Laboratory of Data-Intensive Computing, Key Laboratory of Intelligent Computing and Information Processing, School of Mathematics and Computer Science, Quanzhou Normal University, Quanzhou, China}
\begin{document}
\maketitle
\begin{abstract}
In this work, we revisit the classical Holling type II three species food chain model from a different viewpoint. Two critical parameters $\lambda_1$ and $\lambda_2$ dependent on all  parameters are defined. The existence and local stabilities of all equilibria can be reformulated by $\lambda_1$ and $\lambda_2$, and the complete classifications of parameters and its corresponding dynamics are given. Moreover, with the extinction of top-predator, there is an invariant two dimensional subsystem containing the prey and the intermediate predator. We prove the global stability of the boundary equilibrium in $\mathbb{R}^3_+$ by differential inequality as well as Butler-McGehee lemma if it is stable. Alternatively, there is a unique limit cycle when the boundary equilibrium lost its stability, and we also show the global stability of the limit cycle in $\mathbb{R}^3_+$ by differential inequality and computing the Floquet Multipliers. Finally, some interesting numerical simulations, the chaotic and the bi-stability phenomena, are presented numerically.  A brief discussion and biological implications are also given.
\end{abstract}

\section{Introduction}
In this work, we consider the following three species food chain model with the Holling Type II interaction among the populations,
\begin{align}\label{original_model}
\frac{dX}{dT}=& RX\left(1-\frac{X}{K}\right)-\frac{M_1}{C_1}\frac{XY}{A_1+X},\nonumber \\
\frac{dY}{dT}=& -D_1Y+\frac{M_1XY}{A_1+X}-\frac{M_2}{C_2}\frac{YZ}{A_2+Y}, \\
\frac{dZ}{dT}=& -D_2Z+\frac{M_2YZ}{A_2+Y}, \nonumber \\
X(0) \ge & 0, \ Y(0)\ge 0, \ Z(0) \ge 0, \nonumber
\end{align}
where the species $X$ is a renewable resource, $Y$ is the intermediate predator which predates on $X$, and $Z$ is the top predator which predates on $Y$. The constant $R$ is the intrinsic growth rate; $K$ is the environmental carrying capacity of species $X$; $C_1$ and $C_2$ are conversion rates of prey to predator for species $Y$ and $Z$, respectively; $D_1$ and $D_2$ are constant death rates for species $Y$ and $Z$, respectively; $M_1$, $M_2$, $A_1$ and $A_2$ parametrize the saturating functional response where $A_1$ and $A_2$ are the prey population levels where the predation rate per unit prey is half its maximum value.

To simplify the investigation, we rewrite the model \eqref{original_model} in non-dimensional form. Letting
\begin{equation}\label{rescaling}
\begin{aligned}
& t= RT, \ x=\frac{X}{K},\ y= \frac{M_1Y}{C_1KR},\ z= \frac{M_1M_2Z}{C_1C_2KR^{2}}, \
 m_1= \frac{M_1}{R},\\ 
 & m_2= \frac{M_2}{R},\
 a_1= \frac{A_1}{K},\ a_2= \frac{M_1A_2}{C_1KR},\
 d_1= \frac{D_1}{R},\ d_2= \frac{D_2}{R}, \\
\end{aligned}
\end{equation}
system \eqref{original_model} takes the form
\begin{align}\label{foodchain}
\frac{dx}{dt}=& x\Big(1-x-\frac{y}{a_1+x}\Big),\nonumber \\
\frac{dy}{dt}=& y\Big(-d_1+\frac{m_1x}{a_1+x}-\frac{z}{a_2+y}\Big), \\
\frac{dz}{dt}=& z\Big(-d_2+\frac{m_2y}{a_2+y}\Big), \nonumber
\end{align}
with six parameters $a_i$, $d_i$ and $m_i$ for $i=1, 2$. It is worthy to note that system \eqref{foodchain} can be seen as a combination of two predator prey invariant subsystems, the $x$-$y$ subsystem and the $y$-$z$ subsystem which are well studied. For the two-dimensional $x$-$y$ subsystem, Lyapunov method or phase plane analysis can show a prey-only or coexistence equilibrium of the $x$-$y$ subsystem is globally stable for certain parameters, and, for other parameters, a periodic solution exists. Furthermore, it has been shown that the periodic solution is the unique globally stable limit cycle. Please refer  \cite{Hsu:1978by, Cheng:1981, Cheng:1981ga} and the references cited therein. For the two-dimensional $y$-$z$ subsystem, it is easy to show that $y$ and $z$ approach zero asymptotically.

From another point of view, system \eqref{foodchain} just an invariant $x$-$y$ subsystem coupling with the top predator $z$ which affects only on intermediate predator $y$. 
However, system \eqref{foodchain} is actually a three dimensional system, hence the classical powerful tool, Poincar\'e-Bendixson Theorem, can not be applied to it. So it is difficulty to show global results of system \eqref{foodchain} analytically and this model is also the simplest one of all three trophic level ecosystems with rich dynamics including chaos. Even though system \eqref{foodchain} is investigated over hundreds articles  
on the past fifty years.


In this work, motivated by \cite{Hsu:1978vk},  two critical parameters $\lambda_1$ and $\lambda_2$ are defined by $\lambda_{i}=\frac{a_{i}d_{i}}{m_{i}-d_{i}}$ for $i=1, 2$ which represent the minimum prey population density that can support a given predator for the $x$-$y$ invariant subsystem and the $y$-$z$ invariant subsystem, respectively.  We will use these two parameters to classify the dynamics of \eqref{foodchain} systematically in Table \ref{table1}. Let us present a scenario planning to help us understanding the classifications. First, we show that (Proposition \ref{yzdieout}) if $\lambda_1\ge 1$ then the prey-$x$-only state is global asymptotically stable (GAS) in $\mathbb{R}^3_+$. Otherwise, if $(1 - a_1)/2 < \lambda_1 < 1$ ($x = \lambda_1$ intersects with the falling part of $x$-isocline of $x$-$y$ subsystem), then the coexistence steady state of $x$-$y$ subsystem is globally stable in $x$-$y$ plane. In addition, if 0<$\lambda_1<(1 - a_1)/2$ ($x = \lambda_1$ intersects with the rising part of $x$-isocline of $x$-$y$ subsystem), then the limit cycle is globally stable in $x$-$y$ plane. The key point is what is the dynamics when the top-predator interplays with $x$-$y$ predator prey system? We try to answer this question by the almost necessary and sufficient conditions (\eqref{zdieout-periodic} and \eqref{zsurvive}) of existence of top-predator $z$.

Next, some known global results of system \eqref{foodchain} are reviewed. In 1977, Freedman and Walterman \cite{Freedman:1977vw} show the uniformly persistence of \eqref{foodchain} for the parameters when the boundary invariant subspace, the $x$-$y$ plan, has only the saddle equilibrium without limit cycle. Chiu and Hsu \cite{Chiu:1998}, in 1998, show by Lyapunov method the global asymptotic stability of the boundary equilibrium with extinction of top-predator $z$. In this work, we not only analytically show the previous results \cite{Chiu:1998} again by elementary differential inequalities, but we also prove a novel global result, the global stability of the $x$-$y$ periodic solution with the extinction of the top predator $z$, by the similar method.

When the top-predator $z$ survives, the dynamics of system \eqref{foodchain} is clearly more complex. Logically, there are two cases to be considered where there exists either the stable equilibrium or the stable limit cycle on the boundary $x$-$y$ plane. For the first case, we show that the existence of the positive equilibrium implies the instability of the boundary equilibrium on $x$-$y$ plane by linear method. On the other hand, if there exists a limit cycle on the boundary of the $x$-$y$ plane then the situation is different. Numerically, we find the boundary limit cycle and the positive equilibrium can coexist where we call this phenomenon the point-cycle bi-stability. Even more, if the positive equilibrium is unstable and bifurcates an interior periodic solution then the so-called cycle-cycle bi-stability happens by numerical observations. Parameters of some interesting investigations \cite{Hastings:1991tv, Mccann1995, Kuznetsov1996, Boer:1999gh, Kuznetsov:2001ee} about Hopf bifurcations, homoclinic/hetroclinic bifurcations and chaos for \eqref{foodchain} belong to this case.
The bi-stability phenomena are not rare in ecological models. For example, it is well known that bi-stability occurs in the two competitive model, and recently it is also found in the intraguild predator model \cite{Hsu:2015}. Conventionally, the bi-stability is the so-called point-point bi-stability, which means that there are two boundary stable equilibria separated by an interior saddle equilibrium. However, in this article, we numerically find the point-cycle bi-stability and cycle-cycle bi-stability, that is, we find a stable equilibrium point/cycle and a stable cycle exist simultaneously. The solution of \eqref{foodchain} will approach the stable equilibrium or the stable limit cycle dependent on the initial points. 


Our contributions for this work are following. First, two key parameters $\lambda_1$ and $\lambda_2$ are introduced to completely classify all dynamics of models \eqref{foodchain}. Second, two global stability of boundary equilibrium and one global stability of boundary limit cycle are showed analytically when the top predator is extinct. Thirdly, based on the complete classification, all cases are generically performed numerical simulations beside the proved cases analytically, and the new point-cycle and cycle-cycle bi-stabilities are discovered. Finally, a brief discussion and biological implications are given.

The remainder of this article is organized as follows. In Section 2, we make two assumptions of system \eqref{foodchain} based on two global extinction results. Then we recall all well known results of two-dimensional predator-prey systems. In Section 3, with parameters $\lambda_1$ and $\lambda_2$, the local stabilities of all boundary equilibria in $\mathbb{R}^3$ are discussed and classified, and we also obtain the necessary and sufficient conditions to guarantee the existence and multiplicity of the positive equilibrium. The local stability of coexistence is investigated by the Routh-Hurwitz criterion. With extinction of the top predator, two global stabilities of the boundary equilibrium or the boundary cycle are established by differential inequality and computing the Floquet multipliers. In final section, some numerical simulations are performed for the cases without proof, and a brief discussion as well as some biological interpretations are given.

\section{Preliminary Results}

First of all, we can easily see that the solutions of \eqref{foodchain} with non-negative/positive initial conditions are non-negative/positive. With biological meaningful, the state space of \eqref{foodchain} are restricted on the positive octant,
\[\mathbb{R}^3_+=\left\{(x, y, z)\in\mathbb{R}^3 : x>0, y> 0, z>0\right\}.\]
 Moreover, it can be showed \cite{Freedman:1985eg}, by comparison principle, that all solutions of \eqref{foodchain} initiating in $\mathbb{R}^3_+$ are bounded and eventually enter the attracting set
\begin{align}\label{boundedness}
\left\{(x, y, z)\in\mathbb{R}^3_+: x\le 1,  x+y\le 1+\frac{1}{4d_1},  x+y+z\le 1+\frac{1}{4d_1}+\frac{1}{4d_2}\right\}.
\end{align}

We now present a global extinction result which means that if species $y$ {\it cannot} overcome its natural death rate by getting benefit from species $x$ then it will die out eventually. Consequently, so is species $z$. The proof can be obtained easily by differential inequality, so we omit it.
\begin{lemma}\label{m1d1condition}
If $d_1 \ge m_1$ then $\lim_{t\to\infty}y(t)=0$ and $\lim_{t\to\infty}z(t)=0$
\end{lemma}
Actually, we would like to show a stronger result in the following proposition.
\begin{proposition}\label{yzdieout}
If $d_1\ge\frac{m_1}{a_1+1}$, then
\[ \lim_{t\to\infty} y(t)=0 \ \text{ and }\ \lim_{t\to\infty} z(t)=0.\]
Furthermore, the equilibrium $E_x$ is globally asymptotically stable, in short, GAS.
\end{proposition}
\begin{proof}
We only show that if $d_1\ge\frac{m_1}{a_1+1}$, then $\lim_{t\to\infty}y(t)=0$. Since the extinction of species $y$ implies extinction of species $z$ and the global stability of species $x$, consequently.

We may assume $x(t)\leq1$ for $t$ large enough without loss of generality. Consider the case $d_1>\frac{m_1}{a_1+1}$, and let $\mu_1 = d_1-\frac{m_1}{a_1+1}>0$. Then, by differential inequality, we obtain
\begin{align*}
\frac{\dot y(t)}{y(t)}= -d_1+\frac{m_1x}{a_1+x}-\frac{z}{a_2+y}\leq -d_1+\frac{m_1x}{a_1+x}\leq -d_1+\frac{m_1}{a_1+1}=-\mu_1.
\end{align*}
This inequality implies $\lim_{t\to\infty} y(t)=0$ which implies $\lim_{t\to\infty} z(t)=0$, and system \eqref{foodchain} will asymptotically approach the limiting system \eqref{1d} by the Markus limiting theorem \cite{Markus:1956}. Hence the equilibrium $E_x$ is globally asymptotically stable.

For the case $d_1=\frac{m_1}{a_1+1}$,
\begin{align*}
\dot y= y\Big(-d_1+\frac{m_1x}{a_1+x}-\frac{z}{a_2+y}\Big)\leq y(-d_1+\frac{m_1x}{a_1+x})\leq y(-d_1+\frac{m_1}{a_1+1})=0,
\end{align*}
so $y(t)$ is monotone decreasing. Suppose that $\lim_{t\to\infty}y(t)=\xi>0$, we would like to get a contradiction. By the third equation of \eqref{foodchain}, we have
\begin{align*}
\frac{\dot z(s)}{z(s)}&=-d_2+\frac{m_2y(s)}{a_2+y(s)}\ge -d_2+\frac{m_2\xi}{a_2+\xi},
\end{align*}
which implies that $z(t)\ge z(0)e^{(-d_2+\frac{m_2\xi}{a_2+\xi})t}$. By this inequality, we can see that the inequality $-d_2+\frac{m_2\xi}{a_2+\xi}\le 0$ should be true easily. Otherwise, $z$ will be unbounded which is a contradiction.

 If $-d_2+\frac{m_2\xi}{a_2+\xi}=0$, then it is easy to show that $z$ is monotone increasing. Then
\begin{align*}
\frac{\dot y(t)}{y(t)} & = -d_1+\frac{m_1x}{a_1+x}-\frac{z}{a_2+y}\\
& \leq -d_1+\frac{m_1}{a_1+1}-\frac{z}{a_2+y}= -\frac{z}{a_2+y}\le-\frac{z(0)}{a_2+y(0)}<0,
\end{align*}
which implies $\lim_{t\to\infty}y(t)=0$ and contradicts to $\lim_{t\to\infty}y(t)=\xi>0$.

On the other hand, if $-d_2+\frac{m_2\xi}{a_2+\xi}< 0$, then 
\begin{align*}
\limsup_{t\to\infty}\frac{\dot z(t)}{z(t)}=-d_2+\limsup_{t\to\infty}\frac{m_2y(t)}{a_2+y(t)}=-d_2+\frac{m_2\xi}{a_2+\xi}<0,
\end{align*}
which implies that $\dot z(t)/z(t)$ is less than a negative constant for time large enough. Therefore, we have $\lim_{t\to\infty}z(t)=0$. On the other hand, we also have $\lim_{t\to\infty}\dot y(t)=0$, since $y(t)$ is monotone decreasing to the constant $\xi$. By taking limit of both sides of the second equation of \eqref{foodchain},  we obtain
\[\lim_{t\to\infty}\frac{m_1x(t)}{a_1+x(t)}=d_1 \quad\text{ and  }\quad \lim_{t\to\infty}x(t)=\frac{a_1d_1}{m_1-d_1}=1.\]
Then a contradiction can be obtained by taking limit of both sides of the first equation of system \eqref{foodchain}, and we complete the proof.
\end{proof}

Similarly, by the monotonicity of function $\frac{y}{a_2+y}$, if $d_2\ge m_2$ then we have
\begin{align*}
\frac{\dot z}{z}=-d_2+\frac{m_2y}{a_2+y}\le-d_2+\frac{m_2y_M}{a_2+y_M}<0
\end{align*}
where $y_M=\max_{t\ge0} y(t)$. Consequently, we also have the following lemma.
\begin{lemma} If $d_2\ge m_2$ then $z(t)$ approaches 0 as $t$ approaches to $\infty$.\label{zdieout}
\end{lemma}


By preceding two results, system \eqref{foodchain} will be reduced to a one-dimensional or two-dimensional subsystem for time large enough if $d_1\ge\frac{m_1}{a_1+1}$ or $d_2\ge m_2$, respectively. Hence it is natural to assume that $d_1<\frac{m_1}{a_1+1}$ ( which implies $d_1<m_1$) and $d_2< m_2$ for avoiding these trivialities. In addition, it is clear that $d_1<\frac{m_1}{a_1+1}$ is equivalent to 
\begin{align}\label{lambda1}
0<\lambda_1\equiv\frac{a_1d_1}{m_1-d_1}<1,
\end{align}
where $\lambda_1$ is the first key parameters defined in \cite{Hsu:1978vk}.
Therefore, in the remainder of this work,  let us make these two assumptions,
\begin{description}
\item[{\bf (A1)}] $0<\lambda_1<1$,
\item[{\bf (A2)}] $\displaystyle d_2< m_2$.
\end{description}

\subsection{Dynamics of \eqref{foodchain} on Invariant Subspaces}

It is clear that the system \eqref{foodchain} has three invariant subspaces, $H_1=\{(x,0,0):x\geq0\}$, $H_2=\{(x,y,0):x\ge0,y\ge0\}$, and $H_3=\{(0,y,z):y\ge 0,z\ge0\}$. Furthermore, three boundary equilibria, $E_0=(0, 0, 0)$, $E_x=(1, 0, 0)$ and $E_{xy}=(\bar x_*, \bar y_*, 0)$, can be easily obtained if the assumption {\bf (A1)} holds, where $\bar x_*=\lambda_1$ and $\bar y_*=p(\lambda_1)$ with
\begin{align}\label{poly}
p(x)\equiv (1-x)(a_1+x).
\end{align}
We list all well known global results of system \eqref{foodchain} on these subspaces.
\begin{enumerate}[(i)]
\item On $H_1$, system \eqref{foodchain} is actually one-dimensional system 
\begin{align}\label{1d}
\dot x=x(1-x).
\end{align} 
The equilibrium $E_0$ is unstable, and the equilibrium $E_x$ is GAS on $H_1$.
\item On $H_3$, the equilibrium $E_0=(0,0,0)$  is GAS.
\item On $H_2$, system \eqref{foodchain} can be reduced to the following two-dimensional subsystem
\begin{align}\label{subsystem}
\begin{cases}
\frac{dx}{dt}=x(1-x-\frac{y}{a_1+x}), \\
\frac{dy}{dt}=y(-d_1+\frac{m_1x}{a_1+x}).
\end{cases}
\end{align}
Similarly, on $H_2$, the equilibrium $E_0$ is unstable, and $E_x$ are unstable if $0<\lambda_1<1$.
%
Furthermore, the Jacobian matrix evaluated at $E_{xy}=(\bar x_*,\bar y_*)$  can be obtained by direct computations,
\begin{align*}
A(\bar x_*, \bar y_*)=
\begin{bmatrix}
-\bar x_*+\frac{\bar x_*\bar y_*}{(a_1+\bar x_*)^2} & -\frac{\bar x_*}{a_1+\bar x_*}\\
\frac{a_1m_1\bar y_*}{(a_1+\bar x_*)^2} & 0
\end{bmatrix},
\end{align*}
and the characteristic equation of $A(\bar x_*,\bar y_*)$ is
\[ \lambda^2-\bar x_*(-1+\frac{\bar y_*}{(a_1+\bar x_*)^2})\lambda+\frac{a_1m_1\bar x_*\bar y_*}{(a_1+\bar x_*)^3}=0.\]
Therefore, the equilibrium $E_{xy}$ is locally asymptotically stable on $H_2$, and it is actually global asymptotically stable (GAS) \cite{Hsu:1978by} if
$$-1+\frac{\bar y_*}{(a_1+\bar x_*)^2}<0$$
which is equivalent to
\begin{align}\label{2dGAS}
\bar x_*=\lambda_1>\frac{1-a_1}{2}.
\end{align}
For $a_1\ge1$, it is clear that \eqref{2dGAS} is always true, hence $E_{xy}$ is GAS on $H_2$. However, for $0<a_1<1$ and $\lambda_1=\frac{1-a_1}{2}$, system \eqref{subsystem} happens Hopf bifurcation. If $0<a_1<1$ and $\lambda_1<\frac{1-a_1}{2}$ then $E_{xy}$ becomes an unstable spiral, and there exists a uniqueness stable limit cycle \cite{Cheng:1981ga}.
\end{enumerate}
We summarize all well known one- and two-dimensional results \cite{Bulter:1983, Cheng:1981ga, Cheng:1981} in the following proposition.
\begin{proposition}\label{2dresults}
\begin{enumerate}[\rm (i)]
\item The trivial equilibrium $E_0$ is saddle on $H_1$ and GAS on $H_3$.
\item On $H_2$, the semi-trivial equilibrium $E_x$ is GAS  if $\lambda_1\ge 1$, and it is saddle  if $0<\lambda_1<1$.
\item The equilibrium $E_{xy}$ exists uniquely if and only if $0<\lambda_1<1$, and it is GAS on $H_2$ if \eqref{2dGAS} holds. In particular, for $0<a_1<1$, if $\lambda_1<(1-a_1)/2$, the equilibrium $E_{xy}$ is an unstable focus on the $x$-$y$ plane, and it is surrounded by a unique stable limit cycle $\Gamma$.
\end{enumerate}
\end{proposition}
%


\section{Local and Global Dynamics in $\mathbb{R}^3_+$}
In this section, we investigate local and some global dynamics of \eqref{foodchain} with positive initial conditions. 
First of all, the local stabilities of these three boundary equilibria of system \eqref{foodchain}, $E_0$, $E_x$ and $E_{xy}$, are investigated. Then the existence of positive equilibria $E_*$ will be showed by introducing another key parameter $\lambda_2$ with some constrains, and its corresponding stability is verified by Routh-Hurwitz criterion. Furthermore, complete classifications of dynamics of \eqref{foodchain} with respective to parameters, $\lambda_1$ and $\lambda_2$, are given. Finally, the global stabilities of the boundary equilibrium $E_{xy}$ and the boundary periodic solution $\Gamma$ are proved under different conditions. 




\subsection{Local Stability of Boundary Equilibria in $\mathbb{R}^3$}

It is easy to obtain Jacobian matrix of \eqref{foodchain}, 
\begin{align*}
A(x,y,z)=
\left[
\begin{array}{ccc}
1-2x-\frac{a_1y}{(a_1+x)^2}, & -\frac{x}{a_1+x}, & 0 \\
\frac{a_1m_1y}{(a_1+x)^2}, & -d_1+\frac{m_1x}{a_1+x}-\frac{a_2z}{(a_2+y)^2} , & -\frac{y}{a_2+y}\\
0, &\frac{a_2m_2z}{(a_2+y)^2} , & -d_2+\frac{m_2y}{a_2+y}
\end{array}
\right],
\end{align*}
by direct computations.
\begin{enumerate}[(i)]
\item For equilibrium $E_0=(0,0,0)$:
The Jacobian matrix evaluated at $E_0$ is
\begin{align*}
A(E_0)=
\left[
\begin{array}{rrr}
1, & 0, &  0 \\
0, & -d_1, &  0 \\
0, & 0, & -d_2
\end{array}
\right].
\end{align*}
Hence $E_0$ is a saddle point with two-dimensional stable subspace $H_3$ and one-dimensional unstable subspace $H_1$.
\item For equilibrium $E_x=(1,0,0)$ : The Jacobian matrix evaluated at $E_x$ is
\begin{align*}
A(E_x)=
\left[
\begin{array}{ccc}
-1, & -\frac{1}{a_1+1}, &  0  \\
0, & \frac{m_1}{a_1+1}-d_1, & 0  \\
0, & 0, & -d_2
\end{array}
\right].
\end{align*}
It is easy to see that $A(E_x)$ have two negative eigenvalues, $-1$ and $-d_2$, with the $x$-axis and $z$-axis as their eigen-subspace, respectively. Moreover, there is a positive eigenvalue, $\frac{m_1}{a_1+1}-d_1=\frac{m_1-d_1}{a_1+1}(1-\lambda_1)$, because of assumption {\bf (A1)}. Hence $E_x$ is a saddle point.

\item For equilibrium $E_{xy}=(\bar x_*, \bar y_*, 0)$:
The Jacobian matrix evaluated at $E_{xy}$ is
\begin{align*}
A(E_{xy})=
\left[
\begin{array}{ccc}
-\bar x_*+\frac{\bar x_*\bar y_*}{(a_1+\bar x_*)^2}, & -\frac{\bar x_*}{a_1+\bar x_*}, & 0 \\
\frac{a_1m_1\bar y_*}{(a_1+\bar x_*)^2}, & 0, & -\frac{\bar y_*}{a_2+\bar y_*}\\
0,  &  0, & -d_2+\frac{m_2\bar y_*}{a_2+\bar y_*}
\end{array}
\right]
\end{align*}
with its characteristic polynomial,
\[ \left(\lambda+d_2-\frac{m_2\bar y_*}{a_2+\bar y_*}\right)\left(\lambda^2-\big(-\bar x_*+\frac{\bar x_*\bar y_*}{(a_1+\bar x_*)^2}\big)\lambda+\frac{m_1a_1\bar x_*\bar y_*}{(a_1+\bar x_*)^3}\right)=0. \]
Hence $E_{xy}$ is asymptotically stable if and only if
\begin{align}\label{Exystable}
-\bar x_*+\frac{\bar x_*\bar y_*}{(a_1+\bar x_*)^2}<0 \mbox{ \quad and \quad}
-d_2+\frac{m_2\bar y_*}{a_2+\bar y_*}<0.
\end{align}
Using the equality, $\bar y_*=(1-\bar x_*)(a_1+\bar x_*)$, the first inequality is equivalent to
\begin{align*}
\bar x_*=\lambda_1>\frac{1-a_1}{2}
\end{align*}
which is the same as the case (iii) of two-dimensional subsystem \eqref{subsystem} in Proposition \ref{2dresults}. The second inequality of \eqref{Exystable} is equivalent to
\begin{align}\label{ExyStable-2}
\bar y_*=p(\bar x_*)=p(\lambda_1)<\frac{a_2d_2}{m_2-d_2 }\equiv\lambda_2
\end{align}
where the quadratic polynomial $p(x)$ is defined in \eqref{poly}.
\end{enumerate}
Let us summarize all local stabilities of boundary equilibria as follows.
\begin{proposition}  \label{boundaryequilibria} Let assumptions {\rm \bf(A1)} and {\rm \bf(A2)} hold.
\begin{enumerate}[{\rm (i)}]
\item The trivial equilibrium $E_0$ is a saddle point with two-dimensional stable subspace $H_3$ and one-dimensional unstable subspace $H_1$.
\item The equilibrium $E_x$ is a saddle point with the $x$-axis and $z$-axis as its stable subspace and unstable eigenvector pointed to interior of first octant.
\item The equilibrium $E_{xy}$ is asymptotically stable if inequalities \eqref{2dGAS} and \eqref{ExyStable-2} hold.
\end{enumerate}
\end{proposition}
\subsection{Existence of Coexistence State and its Local Stability}

To find the positive equilibrium $E_*=(x_*,y_*,z_*)$, we should solve the following system
\begin{align}\label{coexistenceq}
0=&1-x-\frac{y}{a_1+x},\nonumber \\
0=& -d_1+\frac{m_1x}{a_1+x}-\frac{z}{a_2+y}, \\
0=& -d_2+\frac{m_2y}{a_2+y}. \nonumber
\end{align}
With assumption {\bf (A2)} and by solving the last equation of \eqref{coexistenceq}, we can easily get $y_*=\lambda_2$ which is defined in \eqref{ExyStable-2}. From the first equation of  \eqref{coexistenceq}, we can obtain $x_*$ by solving the quadratic polynomial
\begin{align}\label{x*equation}
(1-x_*)(a_1+x_*)=p(x_*)=\lambda_2
\end{align}
with the conditions, $x_*<1$ and $y_*=\lambda_2<1+a_1$. Finally, with the preceding $x_*$ and $y_*$ and by solving the second equation of \eqref{coexistenceq}, we obtain
\begin{align}\label{z*formula}
z_*=(\frac{m_1x_*}{a_1+x_*}-d_1)(a_2+y_*)
\end{align}
where $z_*$ is positive if and only if
$x_*>\lambda_1$.
Summarizing the above discussions, the positive equilibrium $E_*$ exists if and only if we can find a positive number $x_*\in(\lambda_1, 1)$ satisfying \eqref{x*equation}. Based on the above discussions, we present a result for existence of positive equilibrium.
\begin{proposition}\label{existence}
Let assumptions  {\bf (A1)} and  {\bf (A2)} hold.
\begin{enumerate}[{\rm (i)}]
\item If $E_*$ exists, then $0<\lambda_2<\min\{1+a_1, \frac{(1+a_1)^2}{4}\}$.
\item If $1\le a_1$, then there is a unique $E_*$ if and only if
$0<\lambda_2\le p(\lambda_1)$.
\item For $0<a_1< 1$, 
\begin{enumerate}[{\rm (a)}]
\item if $\lambda_1\ge\frac{1-a_1}{2}$, then there exists a unique $E_*$ if and only if $0<\lambda_2<p(\lambda_1)$;
\item if $0<\lambda_1<\frac{1-a_1}{2}$, then $E_*$ exists if and only if  $0<\lambda_2\le \frac{(1+a_1)^2}{4}$. Furthermore, $E_*$ is unique either $\lambda_2=\frac{(1+a_1)^2}{4}$ or $\lambda_2\le p(\lambda_1)$, and $E_*$ can be solved exactly with multiplicity two if $p(\lambda_1)<\lambda_2<\frac{(1+a_1)^2}{4}$.
\end{enumerate}
\end{enumerate}
\end{proposition}
\begin{proof}
It is easy to see that the quadratic polynomial $p(x)=(1-x)(a_1+x)$ is concave down with two real roots $1$ and $-a_1$ and maximum $\frac{(1+a_1)^2}{4}$ occurred at $\frac{1-a_1}{2}$. The necessary condition for the existence of $E_*$, the statement of (i), can be verified easily by the above discussion.

The remainder of this proposition, we consider two cases, $1\le a_1$ and $0<a_1<1$. For the case $a_1\ge1$, the function $p(x)$ is monotone decreasing on (0, 1) since $p(x)$ attains its global maxima at $(1-a_1)/2\le 0$. For any given $0<\lambda_1<1$, there is one and only one $x_*\in(\lambda_1, 1)$ such that $p(x_*)=\lambda_2$ if $\lambda_2<p(\lambda_1)$. Please refer the Figure \ref{figure1} (a). This completes the proof of (ii). 

For $0<a_1<1$,  the results of cases $\lambda_1\ge (1-a_1)/2$ and  $\lambda_1< (1-a_1)/2$ can be obtained by similar arguments. Please refer the Figure \ref{figure1} (b) and (c). So we omit the proof of (iii).
\end{proof}

\def\f#1{x^2-(#1)*x+3}
\begin{figure}
\begin{center}
\subfigure[$1\le a_1$]{
\begin{tikzpicture}
	\begin{axis}[
	ticks=none,
	width=190pt,
	axis x line=middle,
	axis y line=center,
	xmin=-2.5,
        xmax=2.5,
        ymin=-0.5,
        ymax=2.5,]
	\addplot+[
	mark=none,
	smooth]
	(\x,{(1-\x)*(1.5+\x)});
	\end{axis}
	\draw[-] (3.557, 1.0) -- (3.557, 0.6)
                      node[below] {$1$};
        \draw[dashed] (2.657, 3.0) -- (2.657, 0.6)
                      node[below] {$\lambda_1$};
        \draw[-] (1.03, 1.0) -- (1.03, 0.6)
                      node[below] {$-a_1$};
        \draw[dashed] (0.5, 2.0) -- (3.857, 2.0)
                      node[right] {$\lambda_2$};
        \draw[dashed, red] (3.1, 2.0) -- (3.1, 0.6)
                      node[below] {$x_*$};
\end{tikzpicture}}
\hspace{0.5cm}
\subfigure[$0<a_1< 1$ and $\lambda_1\ge\frac{1-a_1}{2}$]{
\begin{tikzpicture}
	\begin{axis}[
	ticks=none,
	width=190pt,
	axis x line=middle,
	axis y line=center,
	xmin=-0.5,
        xmax=1.5,
        ymin=-0.1,
        ymax=0.7,]
	\addplot+[
	mark=none,
	smooth]
	(\x,{(1-\x)*(0.3+\x)});
	\end{axis}
	\draw[-] (3.82, 0.8) -- (3.82, 0.4)
                      node[below] {$1$};
        \draw[dashed] (2.857, 2.5) -- (2.857, 0.3)
                      node[below] {$\lambda_1$};
        \draw[-] (0.53, 0.8) -- (0.53, 0.4)
                      node[below] {$-a_1$};
        \draw[dashed] (0.5, 1.7) -- (3.857, 1.7)
                      node[right] {$\lambda_2$};
         \draw[dashed, red] (3.30, 1.8) -- (3.30, 0.4)
                      node[below] {$x_*$};
\end{tikzpicture}}
\hspace{0.5cm}
\subfigure[$0<a_1< 1$ and $0<\lambda_1<\frac{1-a_1}{2}$]{
\begin{tikzpicture}[scale=1.5]
	\begin{axis}[
	ticks=none,
	width=190pt,
	axis x line=middle,
	axis y line=center,
	xmin=-0.5,
        xmax=1.5,
        ymin=-0.1,
        ymax=0.7,]
	\addplot+[
	mark=none,
	smooth]
	(\x,{(1-\x)*(0.3+\x)});
	\end{axis}
	\draw[-] (3.82, 0.65) -- (3.82, 0.4)
                      node[below] {$1$};
        \draw[-] (1.4, 0.65) -- (1.4, 0.4)
                      node[below] {$\lambda_1$};
        \draw[-] (0.53, 0.65) -- (0.53, 0.4)
                      node[below] {$-a_1$};
        \draw[dashed] (1.5, 2.5) -- (3.857, 2.5)
                      node[right] {$\lambda_2$};
         \draw[dashed, red] (1.67, 2.5) -- (1.67, 0.55)
                      node[below] {{\footnotesize $x_*$}};
         \draw[dashed, red] (2.68, 2.5) -- (2.68, 0.55)
                      node[below] {{\footnotesize $x_*$}};
         \draw[dashed] (2.24, 2.7) -- (2.24, 0.55)
                      node[below] {{\footnotesize $\frac{1-a_1}{2}$}};
         \draw[dashed]  (2.8, 2.73)--(1.0, 2.73)
                      node[left] {{\footnotesize $(\frac{1+a_1}{2})^2$}};
\end{tikzpicture}}
\caption{All possible generic cases for the existence of positive equilibrium $E_*$ with parameters $\lambda_1$, $\lambda_2$ and $a_1$.}
\label{figure1}
\end{center}
\end{figure}
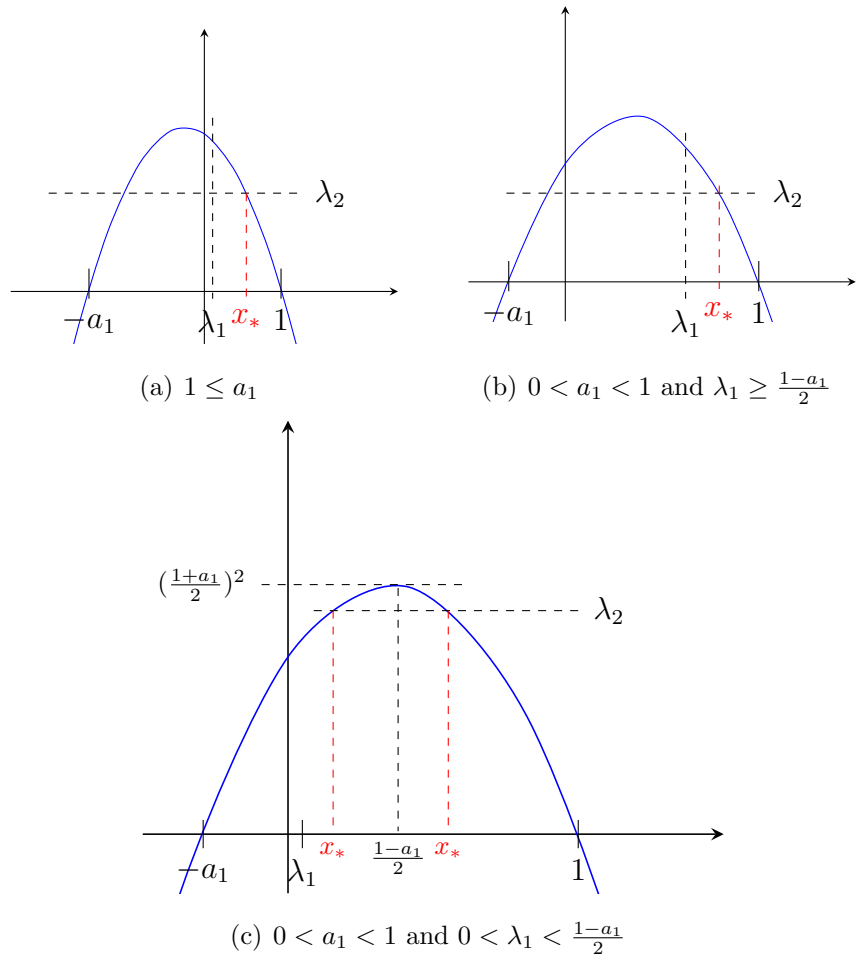


We are in the position to investigate local stability of positive  equilibrium $E_*=(x_*, y_*, z_*)$. By direct computations, the Jacobian evaluated about $E_*$ is
\begin{align}\label{Jac}
J(E_*)=
\left[
\begin{array}{rrr}
-x_*+\frac{x_*y_*}{(a_1+x_*)^2}, & -\frac{x_*}{a_1+x_*}, & 0\\
\frac{a_1m_1y_*}{(a_1+x_*)^2}, & \frac{y_*z_*}{(a_2+y_*)^2}, & -\frac{y_*}{a_2+y_*}\\
0, & \frac{a_2m_2z_*}{(a_2+y_*)^2}, & 0
\end{array}
\right].
\end{align}
To simplify the notations, we set
\begin{align*}
&A = \frac{x_*}{(a_1+x_*)} , &B= \frac{y_*}{(a_1+x_*)},  &&C=\frac{a_1m_1}{(a_1+x_*)},\\
&D = \frac{y_*}{(a_2+y_*)} , &E= \frac{a_2m_2}{(a_2+y_*)}, &&F=\frac{z_*}{(a_2+y_*)},\\
\end{align*}
then \eqref{Jac} can be simplified to the form,
\begin{align*}
J(E_*)=
\left[
\begin{array}{rrr}
-x_*+{AB}, & {-A}, & 0\\
{BC}, & {DF}, & {-D}\\
0, & EF, & 0
\end{array}
\right]
\end{align*}
and the corresponding characteristic equation is
\begin{align}\label{charpoly}
P(\lambda)=\lambda^3+b_2 \lambda^2+b_1 \lambda+b_0=0,
\end{align}
where
\begin{align}\label{RH-1}
b_2 &=x_*-AB-DF,\nonumber \\
b_1 &= (E-x_*)DF+ABDF+ABC,\\
b_0 &= (x_*-AB)DEF.\nonumber
\end{align}
By Routh-Hurwitz criterion, the positive equilibrium $E_*$ is asymptotically stable if and only if  all coefficients, $b_i$, are positive and $b_2b_1>b_0$ where the inequality $b_2b_1>b_0$ is equivalent to
\begin{align}\label{RH-2}
-x_*DF(x_*-AB)-&D^2F^2(E-x_*)+\nonumber \\
&(x_*-AB-DF)(ABDF+ABC)>0.
\end{align}
With the substitution,
\begin{align}\label{yform}
y_*=\lambda_2=(1-x_*)(a_1+x_*),
\end{align}
we can rewrite $x_*-AB$ as the form,
\begin{align*}
x_*-AB = x_* - \frac{x_*y_*}{(a_1+x_*)^2} = x_*(1-\frac{y_*}{(a_1+x_*)^2})
=\frac{x_*}{a_1+x_*}(a_1-1+2x_*).
\end{align*}
Hence $x_*>AB$ if and only if 
\begin{align}\label{RH-3}
x_*>\frac{1-a_1}{2}.
\end{align}
This is a necessary condition of the stability of $E_*$ by \eqref{RH-1}. On the other hand, if  $x_*>AB$ and $F\ll 1$, then it is easy to verified that parameters $b_0$, $b_1$ and $b_2$ are positive, and \eqref{RH-2} holds by taking the limit $F=\frac{z_*}{a_2+y_*}\to 0$, or, equvalently, $z_*\to 0$. By \eqref{z*formula}, $z$ can be seen as a function of $x$, $z(\lambda_1)=0$ and $dz/dx>0$.  Hence we can find a positive number $\varepsilon$ such that $x_*\in(\lambda_1, \lambda_1+\varepsilon)$ implies $E_*$ is stable. 
Let us summarize the results as follows.



\begin{theorem}\label{E*stability}
Let assumptions {\bf (A1)} and  {\bf (A2)} hold and $E_*$ exist. If $E_*$ is asymptotically stable, then \eqref{RH-3} holds. Moreover, if conditions {\bf \eqref{RH-3}} hold, then there exists a positive number $\varepsilon$ such that $x_*\in(\lambda_1, \lambda_1+\varepsilon)$ implies $E_*$ is asymptotically stable.
\end{theorem}

We classify and summarize the existence and stabilities of all equilibria in the Table \ref{table1}. Some global dynamics of system \eqref{foodchain} will be showed in the next subsection, and some remarks and interpretations of Table 1 are given below.
\begin{table}[htp]
    \caption{Classification of parameters by $\lambda_1$ and $\lambda_2$, the details can be found in Remark \ref{E*remark}.}\label{table1}
\begin{center}
\begin{tabular}{|l|c|c|c|c|}
\hline
& $E_x$ & $E_{xy}$ & $E_*$ & Results\\
\hline
(I) $\lambda_1\ge1$ &GAS& $\nexists$ & $\nexists$ & Proposition \ref{yzdieout}\\
\hline
\multicolumn{5}{|l|}{(II) $0<\lambda_1<1$}\\
\hline
\multicolumn{5}{|l|}{\quad (1) $a_1\ge1$($\Rightarrow \lambda_1>\frac{1-a_1}{2}$)} \\
\hline
\quad\quad (a) $\lambda_2> p(\lambda_1)$ && GAS in $\mathbb{R}_+^3$ &$\nexists$ &Theorem \ref{global-results-1}\\
\hline
\multirow{2}{*}{\quad\quad (b) $0<\lambda_2\le p(\lambda_1)$} &&GAS in $H_2$&\multirow{2}{*}{$\exists$!} &Uniformly\\
 &&Saddle in $\mathbb{R}_+^3$ & &Persistence \cite{Freedman1984}\\
\hline
\multicolumn{5}{|l|}{\quad (2) $0<a_1<1$}\\
\hline
\multicolumn{5}{|l|}{\quad\quad (a) $\lambda_1\ge \frac{1-a_1}{2}$} \\
\hline
\quad\quad\quad (i) $\lambda_2 > p(\lambda_1)$ &&GAS in $\mathbb{R}_+^3$&$\nexists$ & Theorem \ref{global-results-1}\\
\hline
\multirow{2}{*}{\quad\quad\quad (ii) $0<\lambda_2<p(\lambda_1)$} &&GAS in $H_2$&\multirow{2}{*}{$\exists$!} &Uniformly\\
 &&Saddle in $\mathbb{R}_+^3$ & &Persistence \cite{Freedman1984}\\
\hline
\multicolumn{5}{|l|}{\quad\quad (b) $\lambda_1< \frac{1-a_1}{2}$}\\
\hline
\multirow{3}{*}{\quad\quad\quad (i) $\lambda_2> (\frac{1+a_1}{2})^2$} && Unstable Spiral in $H_2$& \multirow{2}{*}{$\nexists$} & Theorem \ref{PeriodGAS}\\
&&&&(with \eqref{zdieout-periodic} and \eqref{fcondition}) \\
&&($\exists$! Limit Cycle in $H_2$)&&{\small ($\Gamma$ is GAS)}\\
\hline
\multirow{2}{*}{\quad\quad\quad (ii) $\lambda_2=(\frac{1+a_1}{2})^2$} &&Unstable Spiral in $H_2$& \multirow{2}{*}{$\exists$!} &\multirow{2}{*}{\textcolor{red}{Open}} \\
&&($\exists$! Limit Cycle in $H_2$)&&\\
\hline
\multirow{2}{*}{\quad\quad\quad (iii) $p(\lambda_1)<\lambda_2<(\frac{1+a_1}{2})^2$} &&Unstable Spiral in $H_2$& \multirow{2}{*}{$\exists$2} &\multirow{2}{*}{\textcolor{red}{Open}} \\
&&($\exists$! Limit Cycle in $H_2$)&&\\
\hline
\multirow{2}{*}{\quad\quad\quad (iv) $0<\lambda_2\le p(\lambda_1)$}&&Unstable Spiral in $H_2$& \multirow{2}{*}{$\exists$!} &\multirow{2}{*}{\textcolor{red}{Open}} \\
&&($\exists$! Limit Cycle in $H_2$)&&\\
\hline
\end{tabular}
\end{center}
\end{table}%

\begin{remark}\label{E*remark}
\begin{enumerate}[(i)]
\item In the paper, we always assume that $m_1>d_1$ and $m_2>d_2$, by Proposition \ref{yzdieout} and Lemma \ref{zdieout}, hence $\lambda_i=a_id_i/(m_i-d_i)$ is well defined and positive for $i=1, 2$. Logically, we have two categories, $\lambda_1\ge 1$ \big((I) of Table 1\big) and $\lambda_1<1$ \big((II) of Table 1\big). For $\lambda_1\ge 1$, $E_x$ is GAS by Proposition \ref{yzdieout}. On the other hand, for $\lambda_1<1$, there are two subcases, $a_1\ge1$ \big((II)(1) of Table 1\big) and $0<a_1<1$ \big((II)(2) of Table 1\big). 

\item For the case \big((II)(1) of Table 1\big) $0<\lambda_1<1$ and $a_1\ge 1$ (implied $\lambda_1>(1-a_1)/2$),  then $E_{xy}$ is GAS in $H_2$ by Proposition \ref{2dresults}(iii). With the further condition $\lambda_2>p(\lambda_1)$  \big((II)(1)(a) of Table 1\big), the positive equilibrium $E_*$ does not exist by Proposition \ref{existence}(ii). And Chiu and Hsu \cite{Chiu:1998} showed the global stability of $E_{xy}$ in $\mathbb{R}^3_+$  by modified the Lyapunov function of \cite{Ardito:1995}. Alternatively, we will also give a more simple proof  in Theorem \ref{global-results-1} by comparison principle. On the other hand, in the case \big((II)(1)(b) of Table 1\big), if $\lambda_2< p(\lambda_1)$ then $E_{xy}$ is GAS in $H_2$ as well as saddle in $\mathbb{R}^3_+$ by Proposition \ref{boundaryequilibria}(iii), and we show that there is a unique positive equilibrium $E_*$ by Proposition \ref{existence}(ii)(a). In this case, Freedman and Waltman \cite{Freedman1984} showed the uniform persistence of system \eqref{foodchain}. However, we show that  $E_*$ is stable if $x_*$ is large than and closed to $\lambda_1$ by Theorem \ref{E*stability}, and we conjecture that, for some positive number $\varepsilon$, $E_*$ is GAS for $0<\lambda_1<1$, $a_1\ge 1$ and $0<\lambda_2<p(\lambda_1)$ for $x_*\in (\lambda_1, \lambda_1+\varepsilon)$.
\item The cases (II)(2)(a)(i)-(ii) are similar to cases (II)(1)(a)-(b). The most interesting cases happen in the categories (II)(2)(b) of Table 1. For $\lambda_1<(1-a_1)/2$, the boundary equilibrium $E_{xy}$ of $H_2$ is unstable and there is a unique limit cycle $\Gamma$ in $H_2$ by Proposition \ref{2dresults} \cite{Bulter:1983, Cheng:1981ga, Cheng:1981}. If $\lambda_2>(\frac{1+a_1}{2})^2$ \big((II)(2)(b)(i) of Table 1\big), then the positive equilibrium $E_*$ does not exist by Proposition \ref{existence}(iii)(b) and please refer Figure \ref{figure1}.  In Theorem \ref{PeriodGAS}, we will prove analytically that $\Gamma$ is GAS by applying the Markus Theorem and computing the Floquet multipliers of $\Gamma$.

\item If $\lambda_2\le (\frac{1+a_1}{2})^2$ \big((II)(2)(b)(ii)-(iv) of Table 1\big), then the positive equilibrium $E_*$ exists. In particular, if $p(\lambda_1)<\lambda_2< (\frac{1+a_1}{2})^2$ then we show that there are exactly two positive equilibria with one's $x$-coordinate less than $\frac{1-a_1}{2}$ and another one's $x$-coordinate greater than $\frac{1-a_1}{2}$ by Proposition \ref{existence}(iii)(b), and the positive equilibrium with $x_*<\frac{1-a_1}{2}$ is always unstable by \eqref{RH-3}.
\item We find a sufficient condition, $F=\frac{z_*}{a_2+y_*}\ll 1$, to guarantee the stability of $E_*$. This sufficient condition is reasonable and common in real word. For example, let species $x$, $y$ and $z$ be the plant, herbivore and carnivore, respectively.  It is easy to see that $F\to 0^+$ if and only if $z_*\to 0^+$, which means that the amount of the top predator should be few to stabilize a simple food chain model. This biological implication is compatible with our common sense.
\item It is natural to question that  is $E_*$ globally asymptotically stable when it is stable by Routh-Hurwitz criterion? By some numerical simulations of next section, the answer is Yes or No which is dependent on the dynamics of two-dimensional $x$-$y$ subsystem on $H_2$. 
However, it can be showed that the system is uniformly persistent in the cases (II) (1)(b) and (II)(2)(a)(ii) of Table \ref{table1} by the results in \cite{Freedman1984}, and some details will be discussed in Section 4.

\end{enumerate}
\end{remark}

\subsection{The Global Stability for the Case of Extinction of Top-Predator}
In this subsection, we show analytically the cases (II)(1)(a) and (II)(2)(a)(i) of Table 1, that is, without the top predator $z$, if the boundary equilibrium $E_{xy}$ is stable then it is also GAS. First, we establish the local stability of $E_{xy}$ with the conditions of cases (II)(1)(a) and (II)(2)(a)(i) of Table 1.

\begin{lemma}\label{ExyStable}
If $\lambda_1>\frac{1-a_1}{2}$ and $p(\lambda_1)<\lambda_2$,
then $E_{xy}$ is global asymptotically stable in the interior of positive cone of $x$-$y$ plane and locally asymptotically stable in $\mathbb{R}_+^3$.
\end{lemma}
This can be obtained directly by Proposition \ref{2dresults} (iii) and Proposition \ref{boundaryequilibria} (iii). 

\begin{lemma} Let $\big(x(t), y(t), z(t)\big)$ be a solution of \eqref{foodchain} starting from the interior of positive cone of $\mathbb{R}^3$. With the same assumption of Lemma \ref{ExyStable}, we have 
\begin{align*}
\limsup_{t\to\infty}y(t)\le p(\lambda_1).
\end{align*}
\end{lemma}
\begin{proof}
Let $\big(x(t), y(t), z(t)\big)$ be the solution \eqref{foodchain} starting from the positive initial point $(x_0, y_0, z_0)$, and let $\big(\bar x(t), \bar y(t)\big)$ be the solution of \eqref{subsystem} starting from the positive initial point $(x_0, y_0)$. By comparing $y$-coordinate of vector field of these two models, it is easy to see that
\begin{align*}
-d_1 y+\frac{m_1 x y}{a_1+ x}\ge -d_1y+\frac{m_1 x y}{a_1+ x}-\frac{yz}{a_2+y}.
\end{align*}
Hence we have $\bar y(t)\ge y(t)$ for $t\ge 0$ by differential inequality. Moreover, udner the assumptions, $E_{xy}$ is GAS in the positive cone of $H_2$, hence we have $\lim_{t\to\infty}\bar y(t)=p(\lambda_1)$ by Lemma \ref{ExyStable}, and $\limsup_{t\to\infty}y(t)\le\lim_{t\to\infty}\bar y(t)=p(\lambda_1)$.
\end{proof}

\begin{theorem}\label{global-results-1}
Let $\big(x(t), y(t), z(t)\big)$ be a solution of \eqref{foodchain} starting from the interior of positive cone of $\mathbb{R}^3$. With the same assumption of Lemma \ref{ExyStable}, we have $\lim_{t\to\infty}z(t)=0$ and $E_{xy}$ is global  asymptotically stable in $\mathbb{R}^3_+$.
\end{theorem}
\begin{proof}
Without loss of generality, we may assume that $y(t)\le p(\lambda_1)$ for time large enough. Then 
\begin{align*}
\frac{\dot z}{z}&=-d_2+\frac{m_2y}{a_2+y}\\
&\le -d_2+\frac{m_2p(\lambda_1)}{a_2+p(\lambda_1)}
\end{align*}
which also implies that 
\begin{align*}
\big(a_2+p(\lambda_1)\big)\frac{\dot z}{z}&\le -d_2(a_2+p(\lambda_1))+m_2p(\lambda_1)=(m_2-d_2)p(\lambda_1)-a_2d_2.
\end{align*}
Divided by $m_2-d_2$ to both sides of the last inequality, we obtain
\[\frac{a_2+p(\lambda_1)}{m_2-d_2}\frac{\dot z}{z}\le p(\lambda_1)-\lambda_2<0, \]
which implies that $\lim_{t\to\infty}z(t)=0$.

Applying Markus Limiting Theorem, system \eqref{foodchain} asymptotically approaches 2-dimensional subsystem \eqref{subsystem}. By Lemma \ref{ExyStable}, $E_{xy}$ is global asymptotically stable in the positive cone of $x$-$y$ plane, hence we complete the proof.
\end{proof}

\begin{remark}
In 1998,  Chiu and Hsu \cite{Chiu:1998} considered the following food chain model,
\begin{equation}\label{foodchain-1}
\left\{
\begin{aligned}
\dot X&=RX\left(1-\frac{X}{K}\right)-\frac{M_1XY}{A_1+X},\\
\dot Y&=\left(\frac{M_1X}{A_1+X}-D_1\right)Y-\frac{M_2YZ}{A_2+Y},\\
\dot Z&=\left(\frac{M_2Y}{A_2+Y}-D_2\right)Z,\\
&X(0)>0, Y(0)>0, Z(0)>0,
\end{aligned}
\right.
\end{equation}
which is equivalent to system \eqref{original_model} and \eqref{foodchain} by the same rescaling \eqref{rescaling} with $C_1=C_2=1$. It is easy to see that the two-dimensional subsystem  of \eqref{foodchain-1} only containing species $X$ and $Y$ has a unique positive equilibrium 
\[ (X_*, Y_*)=\left(\frac{A_1D_1}{M_1-D_1}, \frac{R}{M_1}(1-\frac{X_*}{K})(A_1+X_*)\right)\]
which is equivalent to $(\bar x_*, \bar y_*)=\big(\lambda_1, p(\lambda_1)\big)$ of system \eqref{foodchain}. 

Let $\big(X(t), Y(t), Z(t)\big)$ be a solution of system \eqref{foodchain-1}. Chiu and Hsu \cite{Chiu:1998} showed that, by extending the Lyapunov functions introduced by Ardito and Ricciardi \cite{Ardito:1995}, if 
\begin{align}\label{HsuCondition}
 X_*\ge \frac{K-A_1}{2} \quad\text{ and }\quad \frac{M_2Y_*}{A_2+Y_*}-D_2\le 0 
\end{align} 
then $(X(t), Y(t), Z(t))\to (X_*, Y_*, 0)$ as $t\to\infty$ where the inequalities \eqref{HsuCondition} are equivalent to $\lambda_1\ge\frac{1-a_1}{2}$ and $p(\lambda_1)\le\lambda_2$. Hence their results cover the cases (II)(1)(a) and (II)(2)(a)(i) of Table \ref{table1}. Comparing the results, we show the same results of cases (II)(1)(a) and (II)(2)(a)(i) in Theorem \ref{global-results-1} except for the equality sign by a more simple method. Moreover, our method with a little modified can also cover the case (II)(2)(b)(i) with a sufficient condition in the next subsection.
\end{remark}

\subsection{Global Stability of the Limit Cycle $\Gamma$ on $H_2$}
Now we consider the case (II)(2)(b)(i) of Table \ref{table1}. A known result for this case is recalled \cite{Cheng:1981ga}.

\begin{lemma}[Theorem 1, \cite{Cheng:1981ga}]\label{Cheng}
Let $0<a_1<1$ and $\lambda_1<\frac{1-a_1}{2}$, then there exits a unique limit cycle $\Gamma$ of \eqref{subsystem} on $H_2$.
\end{lemma}

In the next result, we will show the local stability of the previous unique limit cycle $\Gamma$ in $\mathbb{R}^3_+$ by considering the linearizing system of \eqref{foodchain} about $\Gamma$ and computing the Floquet multipliers \cite[page 112]{Iooss1992}.
\begin{proposition}\label{stabilityperiod}
With the same assumptions of Lemma \ref{Cheng}, the periodic solution $\Gamma$ is locally asymptotically stable in $\mathbb{R}^3_+$ if 
\begin{align}\label{zdieout-periodic} 
-d_2+\frac{1}{T}\int_0^T \frac{m_2\tilde y(t)}{a_2+\tilde y(t)} \ dt<0, 
\end{align}
where $\Gamma$ is the unique limit cycle on $H_2$ in Lemma \ref{Cheng}, $T$ is the period of $\Gamma$, and $\tilde y(t) $ is the $y$-coordinates of $\Gamma$, that is, $\Gamma=\left\{\gamma(t)=\big(\tilde x(t), \tilde y(t), 0\big): t\in[0, T]\right\}$. 
On the other hand, if 
\begin{align}\label{zsurvive}
-d_2+\frac{1}{T}\int_0^T \frac{m_2\tilde y(t)}{a_2+\tilde y(t)} \ dt>0, 
\end{align}
then $\Gamma$ is unstable in $\mathbb{R}^3_+$.
\end{proposition}
\begin{proof}
Let $\xi\in\mathbb{R}^3$ be small enough with positive $z$-coordinate as well as $\xi+\gamma(t)\in\mathbb{R}^3_+$ for all $t\in[0, T]$. And $\phi(t; \xi+\gamma(0))$ is the solution of \eqref{foodchain} with initial point $\xi+\gamma(0)$. Note that $\phi(t; \gamma(0))=\gamma(t)$. Define
\begin{align*}
h(t)
\equiv\phi(t; \xi+\gamma(0))-\phi(t; \gamma(0)),
\end{align*}
then we have the estimations 
\begin{align}\label{approximation}
\frac{dh}{dt}
&=\frac{d}{dt}\phi(t; \xi+\gamma(0))-\frac{d}{dt}\phi(t; \gamma(0))\nonumber \\
&=F(\phi(t; \xi+\gamma(0)))-F(\phi(t; \gamma(0)))\nonumber \\
&=DF(\phi(t; \gamma(0)))h(t)+ o(||h(t)||),\nonumber \\
&=DF(\gamma(t))h(t)+ o(||h(t)||)
\end{align}
for $||\xi||$ small enough and $F:\mathbb{R}^3\to \mathbb{R}^3$ is the right hand side function of \eqref{foodchain}. 

By considering the linear part of the previous equation,
\[\frac{dh}{dt}=DF(\gamma(t))h(t),\]
it is easy to see that the fundamental matrix solution $M(t)=\left[M_{i,j}(t)\right]_{i, j=1}^3$ satisfies the equation
\begin{align}\label{3d-linearperiodicsystem}
\frac{dM}{dt}
 &=\begin{bmatrix} 1-2\tilde x-\frac{a_1\tilde y}{(a_1+\tilde x)^2} & -\frac{\tilde x}{a_1+\tilde x}&0\\ 
 \frac{a_1m_1\tilde y}{(a_1+\tilde x)^2} & -d_1+\frac{m_1\tilde x}{a_1+\tilde x}& -\frac{\tilde y}{a_2+\tilde y}\\ 
 0 & 0 & -d_2+\frac{m_2\tilde y}{a_2+\tilde y}
 \end{bmatrix} M.
\end{align}
So we can easily solve the components, $M_{31}(t)\equiv0$ and $M_{32}(t)\equiv0$. Moreover, $M_{33}$, satisfying the equation
\[\frac{d}{dt} M_{33}=\left(-d_2+\frac{m_2\tilde y(t)}{a_2+\tilde y(t)}\right)M_{33}(t),\]
can be solved as
\begin{align}
M_{33}(t)=\exp\left({\displaystyle\int_0^t-d_2+\frac{m_2\tilde y(s)}{a_2+\tilde y(s)}ds}\right).
\end{align}
Hence the monodromy matrix $M(T)$ has the form,
\begin{align}
M(T)=
\begin{bmatrix}\label{monodromy}
M_{11}(T) & M_{12}(T) & M_{13}(T)\\
M_{21}(T) & M_{22}(T) & M_{23}(T)\\
0 & 0 & M_{33}(T)
\end{bmatrix},
\end{align}
and the local stability of $\Gamma$ in $\mathbb{R}^3_+$ is dependent on two Floquet multipliers (eigenvalues) corresponding to the up-left $2\times 2$ submatrix of \eqref{monodromy},
\begin{align}
\begin{bmatrix}\label{2dmonodromy}
M_{11}(T) & M_{12}(T) \\
M_{21}(T) & M_{22}(T) 
\end{bmatrix},
\end{align}
 as well as the third Floquet multiplier, $M_{33}(T)$. 

In addition, the matrix \eqref{2dmonodromy} is exact the monodromy matrix of $\Gamma$ on $H_2$ which is an invariant subspace of $\mathbb{R}^3_+$, and, under assumptions of Lemma \ref{Cheng}, $\Gamma$ is the limit cycle on $H_2$. Hence we can establish that the local stability of $\Gamma$ in $\mathbb{R}^3_+$ is dependent on $M_{33}(T)$ only. So if  \eqref{zdieout-periodic} holds then $\Gamma$ is asymptotically stable, and, on the other hand, if  \eqref{zsurvive} holds, then $\Gamma$ is unstable. This completes the proof.

\end{proof}
Before we show the second main global result, an extinction result of species $z$ will be established with a sufficient condition. All solutions $\big(x(t), y(t), z(t)\big)$ of \eqref{foodchain} with a positive initial point  eventually enter the bounded attracting set \eqref{boundedness}. We may assume $y_M=\max y(t)$ for $t$ large enough.
\begin{lemma}\label{z-dieout}
Let $\big(x(t), y(t), z(t)\big)$ be a solution of \eqref{foodchain} with a positive initial point. With the same assumptions of Proposition \ref{stabilityperiod}, if the inequality
\begin{align}\label{fcondition}
\frac{a_2}{a_2+y_M}\lambda_2>\frac{(1+a_1)^3}{4a_1}
\end{align}
holds, then $\lim_{t\to\infty}z(t)=0$.
\end{lemma}
\begin{proof}
Let $\eta\equiv\frac{a_2}{a_2+y_M}\lambda_2-\frac{(1+a_1)^3}{4a_1}>0$. Without loss of generality, we assume that $0<x(t)\le 1$ and $0<y(t)\le y_M$ for $t\ge 0$. Then the first and third equations of \eqref{foodchain} can be rewritten as the forms,
\begin{align*}
\frac{\dot x}{x}&=\frac{1}{a_1+x}\big(p(x)-y\big),\\
\frac{\dot z}{z}&=\frac{m_2-d_2}{a_2+y}(y-\lambda_2).
\end{align*}
Furthermore, the value $\left(\frac{1+a_1}{2}\right)^2$ is the global maximum of $p(x)$ occurred at $\frac{1-a_1}{2}$, then
\begin{align*}
\frac{a_2}{m_2-d_2}\frac{\dot z}{z}+(1+a_1)\frac{\dot x}{x}&=\frac{a_2}{a_2+y}(y-\lambda_2)+\frac{1+a_1}{a_1+x}p(x)-\frac{1+a_1}{a_1+x}y\\
&\le \frac{1+a_1}{a_1+x}p(x)-\frac{a_2}{a_2+y}\lambda_2\\
&\le \frac{1+a_1}{a_1}\left(\frac{1+a_1}{2}\right)^2-\frac{a_2}{a_2+y_M}\lambda_2=-\eta.
\end{align*}
This inequality implies that $Kx(t)^{1+a_1}z(t)^{\frac{\lambda_2}{d_2}}\le -\eta t$ for some positive constant $K$ by integrating both side from 0 to $t$, and, consequently, 
\begin{align}\label{approach}
x(t)^{1+a_1}z(t)^{\frac{\lambda_2}{d_2}}\to 0 \text{ as } t\to\infty.
\end{align}

Logically, we only need to consider two cases. The one is that $x(t)$ is bounded below by a positive number for all time large enough. The other one is that there is a sequence of time $\{t_k\}$ such that $x(t_k)\to 0$ as $k\to\infty$. If the first one happens, then \eqref{approach} implies $\lim_{t\to\infty}z(t)=0$. 


We assume that the second case happens. Let $\phi(t; p)=\big(x(t), y(t), z(t)\big)$ be the solution of \eqref{foodchain} with positive initial point $p$. The assumption, $x(t_k)\to 0$ as $k\to\infty$, implies that there a point $q_1$ on $y$-$z$ plane belonging to the omega limit set of $p$, $\omega(p)$. Moreover, by Proposition \ref{boundaryequilibria} (i) and invariance of omega limit set, we have $E_0\in\omega(p)$. However, $\phi(t; p)$ does not converge to $E_0$, hence there is a point $q_2\in H_1$ such that $q_2\in\omega(p)$ by Butler-McGehee Lemma \cite{Freedman1984}. Similarly, there is a point $q_3\in H_2$ such that $q_3\in\omega(p)$. Finally, by the global stability of $\Gamma$ on $H_2$, the point $q_3$ approaches $\Gamma$ and $\Gamma\subset\omega(p)$. Since the initial point $p$ is arbitrary, we have that $\Gamma$ belong to the omega limit set of any solution of \eqref{foodchain} with positive initial condition. Consequently, all solutions converge to $\Gamma$ since it is asymptotically stable. This proves that $z(t)\to 0$ as $t\to\infty$ for the case two. So we completes the proof.
\end{proof}
\begin{remark}
\begin{enumerate}[(i)]
\item It is clear that \eqref{fcondition} is a sufficient condition of the inequality $\lambda_2>(1+a_1)^2/4$. We will establish the global stability of the periodic solution $\Gamma$ in $\mathbb{R}^3_+$ under the condition \eqref{fcondition} in the next main result. That is, we show a partial results of the case (II)(2)(b)(i) of Table 1.
\item To verify inequality \eqref{fcondition}, the value $y_M$ is seems dependent on the initial points. However, we can take the upper bound $y_M\le 1+1/(4d_1)$ by the eventual attracting region \eqref{boundedness} which is valid for all solutions starting from any point in $\mathbb{R}^3_+$. Here is a set of parameter,
\begin{itemize}
\item $a_1 = 0.24$, $m_1 = 0.5$, $d_1 = 0.3$, $a_2 = 0.4$, $m_2 = 0.4$, $d_2 = 0.39$,
\item $\lambda_1=a_1d_1/(m_1-d_1)=0.36$, $\lambda_2= a_2d_2/(m_2-d_2)= 15.6$,
\item $(1-a_1)/2= 0.38$, $(1+a_1)^2/4= 0.3844$, $\frac{(1+a_1)^3}{4a_1}\approx 1.986$
\item $1+1/(4d_1)\approx 1.833$,
\item $\frac{a_2}{a_2+1.833}\lambda_2\approx 2.794$,
\end{itemize}
satisfying the inequality \eqref{fcondition}.
\end{enumerate}
\end{remark}
\begin{theorem}\label{PeriodGAS}
With assumptions, $0<a_1<1$, $\lambda_1<(1-a_1)/2$, \eqref{zdieout-periodic} and \eqref{fcondition} hold, the periodic solution $\Gamma$ of \eqref{foodchain} on $H_2$ is globally asymptotically stable in $\mathbb{R}^3_+$.
\end{theorem}
\begin{proof}
Let $\phi(t; p)=\big(x(t), y(t), z(t)\big)$ be the solution of \eqref{foodchain} with positive initial point $p$. With the assumptions, we have $\lim_{t\to\infty}z(t)=0$ by Lemma \ref{z-dieout}. So we can find a point $q\in H_2$ such that $q\in\omega(p)$. Since $\Gamma$ is the unique limit cycle on $H_2$ and the invariance of omega limit set,  we obtain that $\Gamma\subset\omega(p)$. That is, for any $p\in\mathbb{R}^3_+$, there is a sequence of time $\{t_n\}$ such that the solution $\phi(t_n; p)\to \Gamma$ as $n\to\infty$. Consequently, all solutions converge to $\Gamma$ since it is asymptotically stable in $\mathbb{R}^3_+$ by Proposition \ref{stabilityperiod}. The proof is completed.
\end{proof}

\section{Brief Discussions, Biological Implications and Numerical Simulations}


In this work, we investigate the well studied food chain model \eqref{original_model}. After rescaling to \eqref{foodchain}, six parameters, $a_i$, $d_i$ and $m_i$ for $i=1, 2$, can be reformed to two critical parameters $\lambda_1$ and $\lambda_2$ defined firstly in \cite{Hsu:1978vk} which represent the minimum prey population densities, $x$ and $y$, that can support given predators, $y$ and $z$, respectively. Then all well known two-dimensional results are recalled, Proposition \ref{2dresults}, including the extinction of the predator, the existence of positive equilibrium and the existence and uniqueness of the periodic solution, where these facts help us to identify the dynamics of the three dimensional system. \medskip

Based on two extinction results, Proposition \ref{yzdieout} and Lemma \ref{zdieout}, we make assumptions {\bf (A1)} and {\bf (A2)}.  With these two assumptions, parameters $\lambda_1$ and $\lambda_2$ can be used to classify the existence  and dynamics of all equilibria of \eqref{foodchain} in Proposition \ref{boundaryequilibria}, Proposition \ref{existence}, and summarize in Table \ref{table1}. Furthermore, two global extinction/stabilitiy results, Theorem \ref{global-results-1} and Theorem \ref{PeriodGAS} covered cases (II)(1)(a), (II)(2)(a)(i) and (II)(2)(b)(i) of Table \ref{table1}. In particular, it was proved by the differential inequality and the Bulter-McGehee Lemma in Theorem \ref{global-results-1} which is an alternative proof comparing with \cite{Chiu:1998} showed by Lyapunov method for the extinction of species $z$ and the global stability of equilibrium $E_{xy}$. In addition, in Theorem  \ref{PeriodGAS}, we also show the extinction of species $z$, by the differential inequality, and the global stability of the limit cycle on $x$-$y$ plane, by computing the Floquet multipliers. In our knowledge, this is a novel result of food chain model with Holling type II functional response. 

In the cases (II)(1)(b) and (II)(2)(a)(ii) of Table \ref{table1}, although $E_{xy}$ is GAS on $H_2$, it is unstable in $\mathbb{R}^3_+$. Since the inequality $\lambda_2<p(\lambda_1)$ implies that existence of $E_*$ and $E_{xy}$ is unstable by \eqref{ExyStable-2}. So it is naturally to query that is $E_*$ stable or even GAS  in $\mathbb{R}^3_+$? By Theorem \ref{E*stability}, $E_*$ is stable when  $x_*$ is close to $\lambda_1$, and we conjecture, by numerical simulations, that it is GAS. However, it will be an unstable spiral and a periodic solution appear when $x_*$ is far away from $\lambda_1$ by numerical observations. Actually, the uniformly persistence of cases (II)(1)(b) and (II)(2)(a)(ii) of Table \ref{table1} can be proved easily by Theorem 5.1 of  \cite{Freedman1984}. \medskip

Finally, by performing numerical simulations with \verb-xppaut-\cite{ermentrout2002simulating} via \verb-Python- interface, let us discuss some observed numerical phenomena. With Table \ref{table1} as the blueprint, generically, there are only two cases,   (II)(2)(b)(iii) and (II)(2)(b)(iv), to be considered. 
The second one, including cases (II)(2)(b)(iii)- (II)(2)(b)(iv), is that the boundary equilibrium $E_{xy}$ is unstable spiral and there is the unique limit cycle on $H_2$.

We set parameters, $a_1=0.3$, $m_1=5/3$, $d_1=0.4$, $a_2=0.9$, $d_2=0.01$, and vary $m_2$ from 0.02 to 0.15 with step-size 0.001. The values,
\begin{align*}
\frac{1-a_1}{2}&=0.35,\qquad \lambda_1\approx 0.09473684210526316,\\
 \left(\frac{1+a_1}{2}\right)^2&=0.4225, \quad  \text{ and } \quad p(\lambda_1)\approx0.3573407202216067,
\end{align*}
are fixed, and only $\lambda_2$ are various with respect to $m_2$. It can be easily obtained that $\lambda_2\approx (1+a_1)^2/4$ for $m_2\approx 0.0313$ and $\lambda_2\approx p(\lambda_1)$ for $m_2\approx 0.0351$. Some specific values of $m_2$ are taken to simulate the dynamics of system \eqref{foodchain}. Please refer the case (c) of Figure \ref{figure1} and the following Table \ref{table2}. 

\begin{table}[htp]
\caption{Set fixed parameters, $a_1=0.3$, $m_1=5/3$, $d_1=0.4$, $a_2=0.9$, $d_2=0.01$, and vary $m_2$ with different values in the first row. The corresponding $\lambda_2$ and its classification by Table \ref{table1} are showing in the row 2 and 3, respectively.}\label{table2}
\begin{center}
\begin{tabular}{|c|c|c|c|c|c|}
\hline
$m_2$  &0.033&0.042&0.065\\
\hline
$\lambda_2$ &$\approx$ 0.3913&$\approx$0.28&$\approx$0.164\\
\hline
Classification &(II)(2)(b)(iii)& \multicolumn{2}{|c|}{(II)(2)(b)(iv)} \\
\hline
\end{tabular}
\end{center}
\end{table}%

For the the cases, (II)(2)(b)(iii) and (II)(2)(b)(iv) of Table \ref{table1}, on the boundary $H_2$ the equilibrium $E_{xy}$ is unstable spiral and there is the stable periodic solution. What is the global dynamics of system \eqref{foodchain} effecting by the interplay between the boundary periodic solution and the interior equilibrium or interior periodic solution ? Let us see more numerical simulations in the following.

\begin{enumerate}[(i)]

\item $m_2=0.033$ : For this case, in fact, the positive equilibrium can be solved with multiplicity 2, and we get the explicit form for the positive equilibrium
\begin{align*}
x_*&= \frac{1-a_1\pm \sqrt{(1+a_1)^2-4\lambda_2}}{2}\approx 0.1734, 0.5266, \\
y_* &= \lambda_2 \approx 0.3913, \\
z_*&=(a_2+y_*)\Big(\frac{m_1x_*}{a_1+x_*}-d_1\Big)\approx 0.2717, 0.8546.
\end{align*}
Recall from \eqref{RH-3} that the necessary condition for the stability of $E_*$ is $x_*>\frac{1-a_1}{2}$. Hence the equilibrium $(0.1734, 0.3913, 0.2717)$ is unstable. Moreover, we numerically check the stability of equilibrium $(0.5266, 0.3913, 0.8546)$ by Routh-Hurwitz criterion, \eqref{RH-1} and \eqref{RH-2}, and obtain that it is stable. For these parameters, there are a stable positive equilibrium and one limit cycle on $H_2$ which we do not know the stability in $\mathbb{R}_+^3$. So here is an interesting question that how to interplay between these two invariant sets ? However, it is difficulty to check the stability of the limit cycle on $H_2$, so we perform some numerical simulations instead. Two trajectories are simulating by setting parameters which previously mentioned and $m_2=0.033$ with two different initial points, (0.5266,  0.3913, 0.8546) and (0.1734,  0.3913, 0.2717), which are close two positive equilibria, respectively. By referring Figure \ref{figure3}, panels (a) and (b) are the time course of all species and trajectory in $\mathbb{R}^3$ respectively with initial point  (0.5266,  0.3913, 0.8546), and panels (c) and (d) are similar with initial point (0.1734,  0.3913, 0.2717). It is clear that one trajectory converges to the stable positive equilibrium and one trajectory converges to the limit cycle on $H_2$. We plot these two trajectories simultaneously in the panel (e). This phenomenon can be seen as a {\it cycle-point bi-stability}. 

\begin{figure}[hp]
\begin{center}
\subfigure[time courses of all species]{
\includegraphics[scale=0.25]{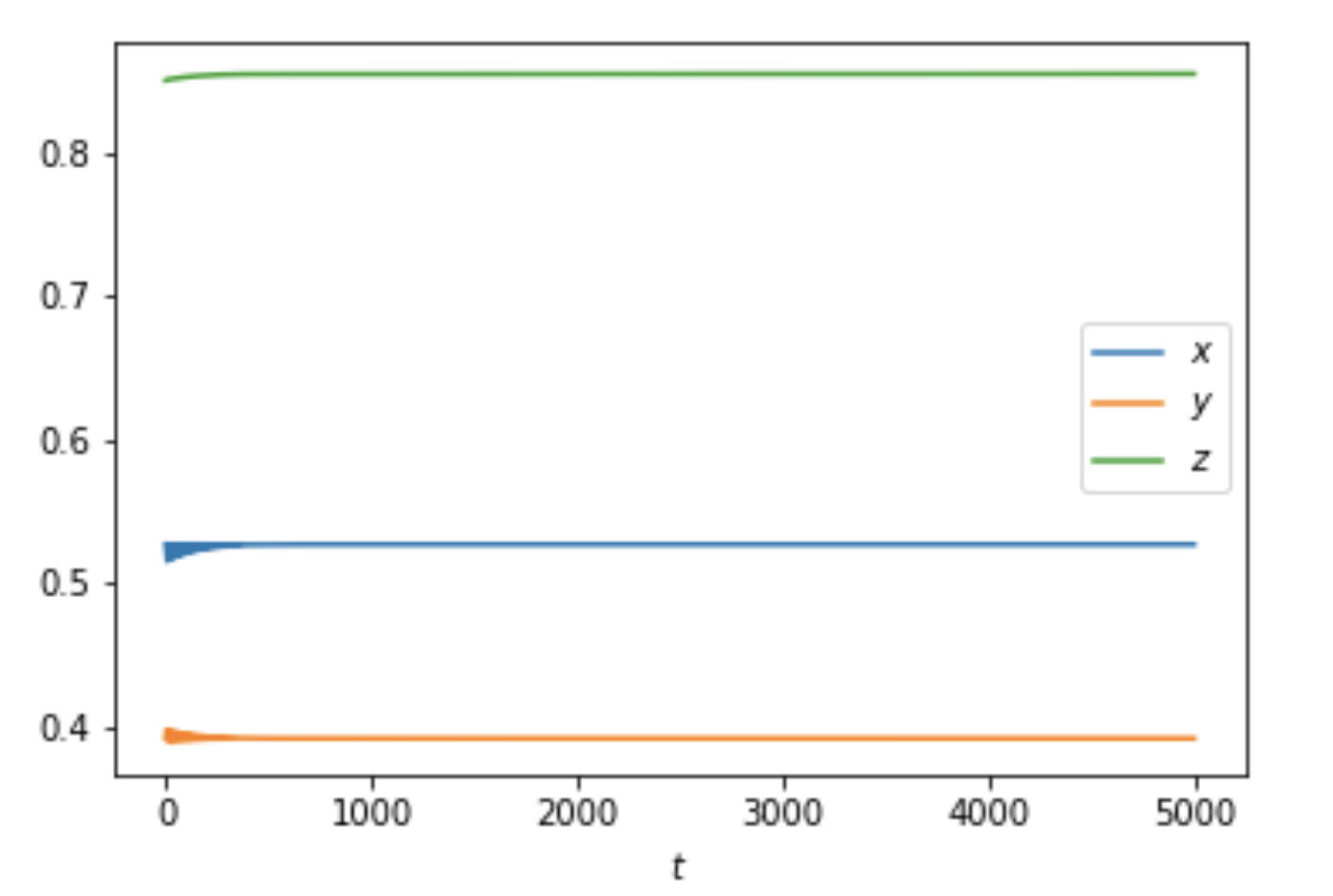}
}
\subfigure[dynamics of \eqref{foodchain}]{
\includegraphics[scale=0.25]{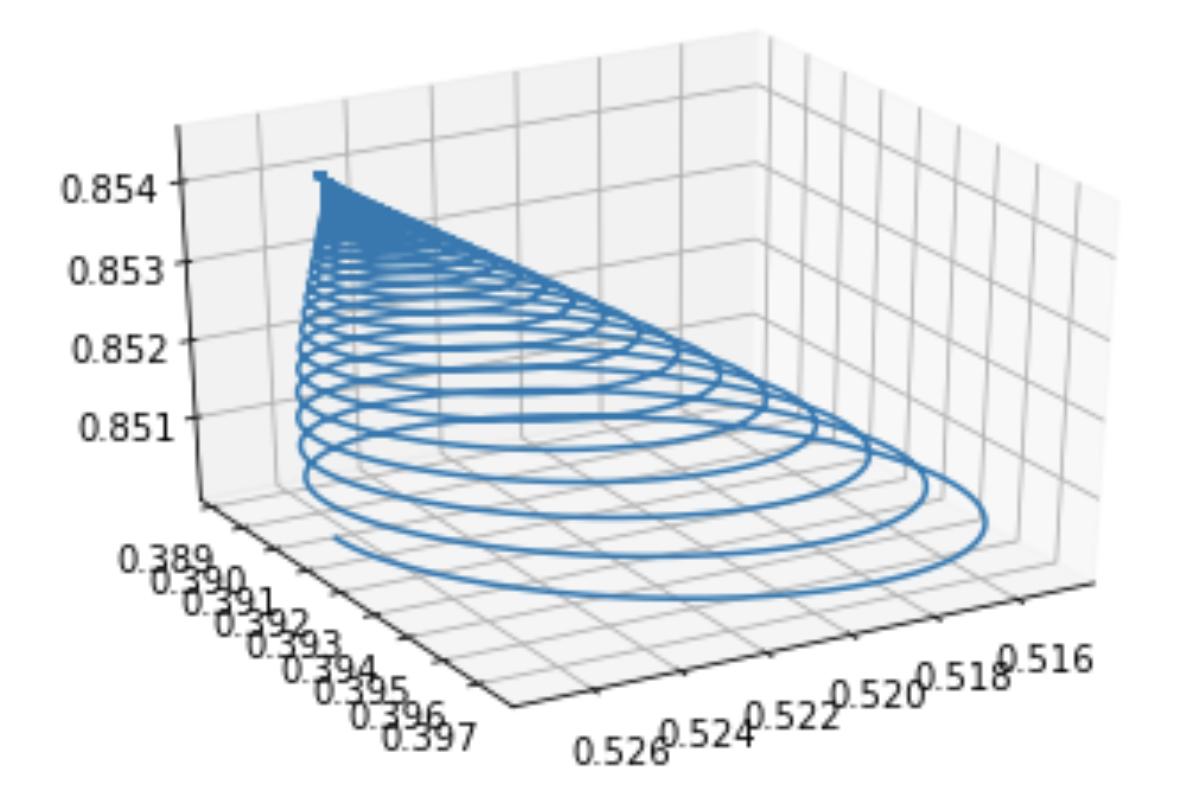}
}
\subfigure[time courses of all species]{
\includegraphics[scale=0.25]{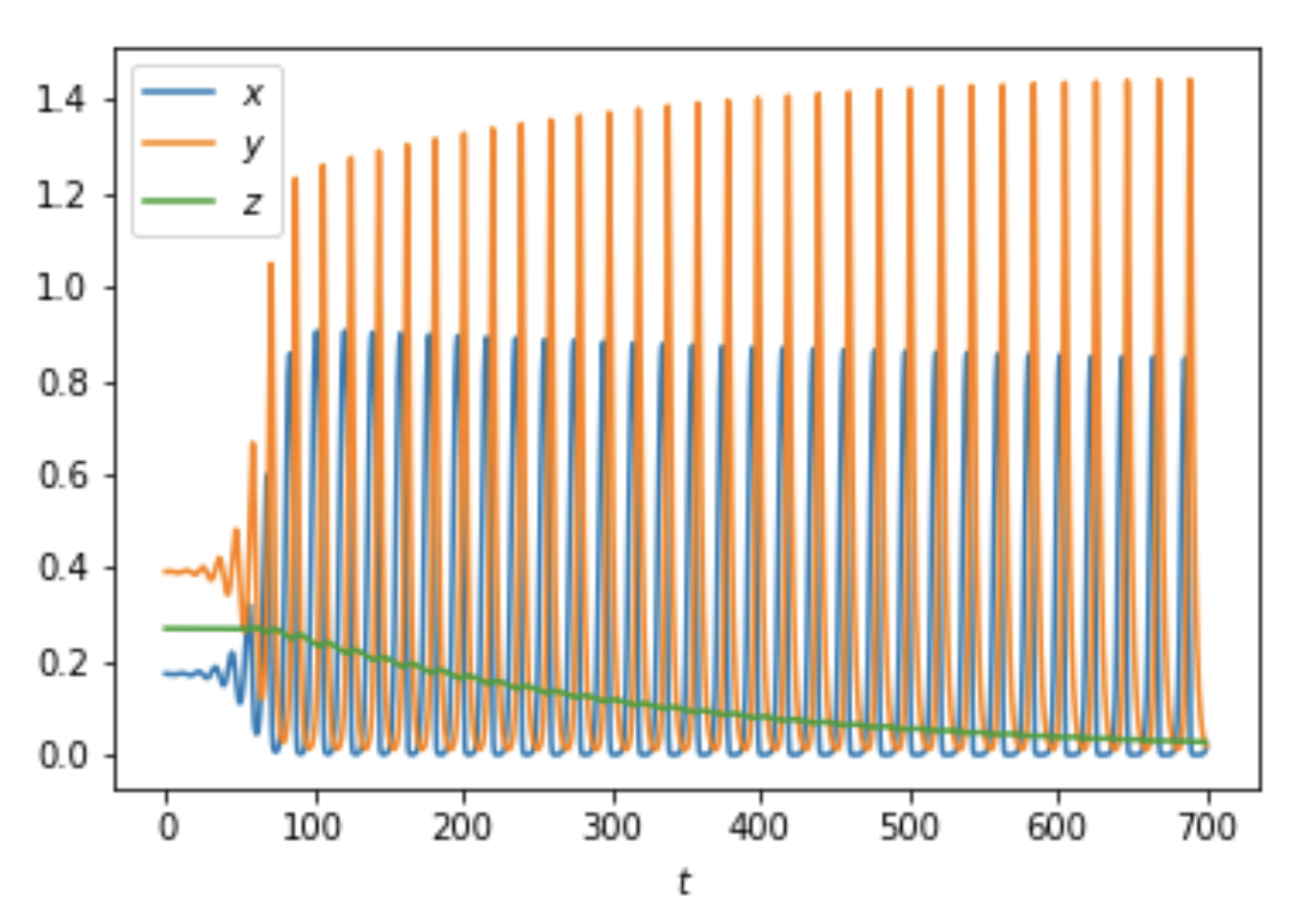}
}%
\subfigure[dynamics of \eqref{foodchain}]{
\includegraphics[scale=0.25]{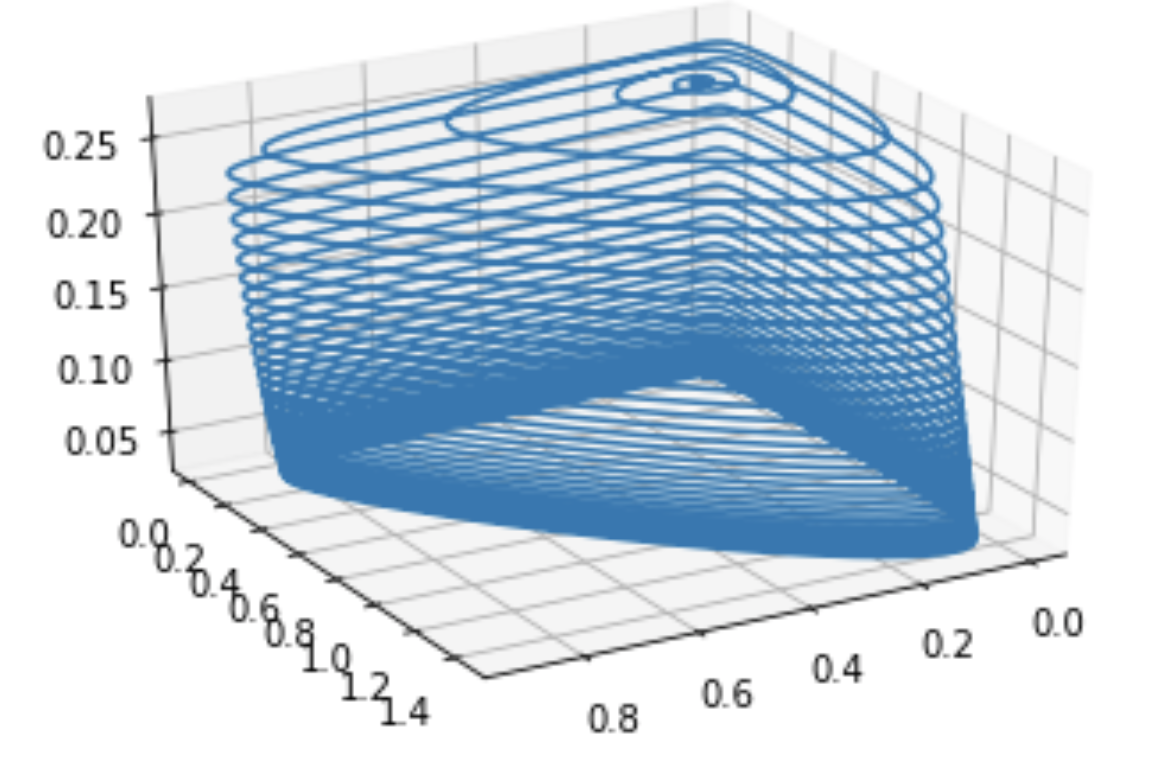}
}
\subfigure[dynamics of \eqref{foodchain}]{
\includegraphics[scale=0.3]{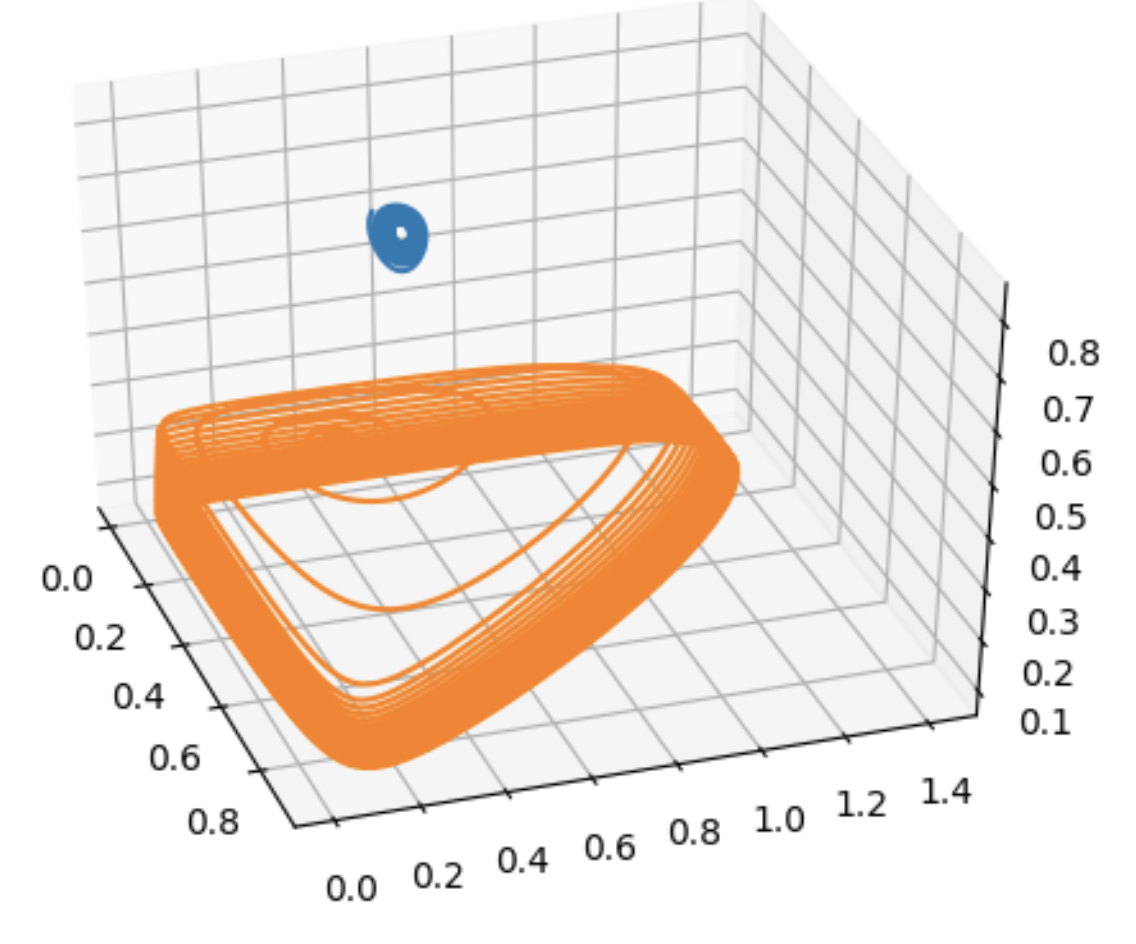}
}
\end{center}
\caption{Set parameters, $a_1 = 0.3$, $m_1 = 5/3$, $d_1 = 0.4$, $a_2 = 0.9$, $d_2 = 0.01$, and $m_2=0.033$. There is a stable interior equilibrium and a limit cycle on $H_2$. The behavior of asymptotic dynamics is dependent on the initial condition. Please see the details in the context.}\label{figure3}
\end{figure}

\item $m_2=0.042$ : For this case, there is only one positive equilibrium which is unstable by using the approximating point (0.7472, 0.2647, 0.9192) to check the stability  via Routh-Hurwitz criterion, \eqref{RH-1} and \eqref{RH-2}. Similarly, we simulate two trajectories of \eqref{foodchain} whose initial points are close to the positive equilibrium (panels (a) and (b) of Figure \ref{figure4}) and close to $H_2$  (panels (c) and (d) of Figure \ref{figure4}), respectively. Putting these two trajectories simultaneously  (panel (e) of Figure \ref{figure4}), we also get a bi-stability phenomenon between two stable periodic solutions where we call it a {\it cycle-cycle bi-stability}.

\begin{figure}[hp]
\begin{center}
\subfigure[time courses of all species]{
\includegraphics[scale=0.25]{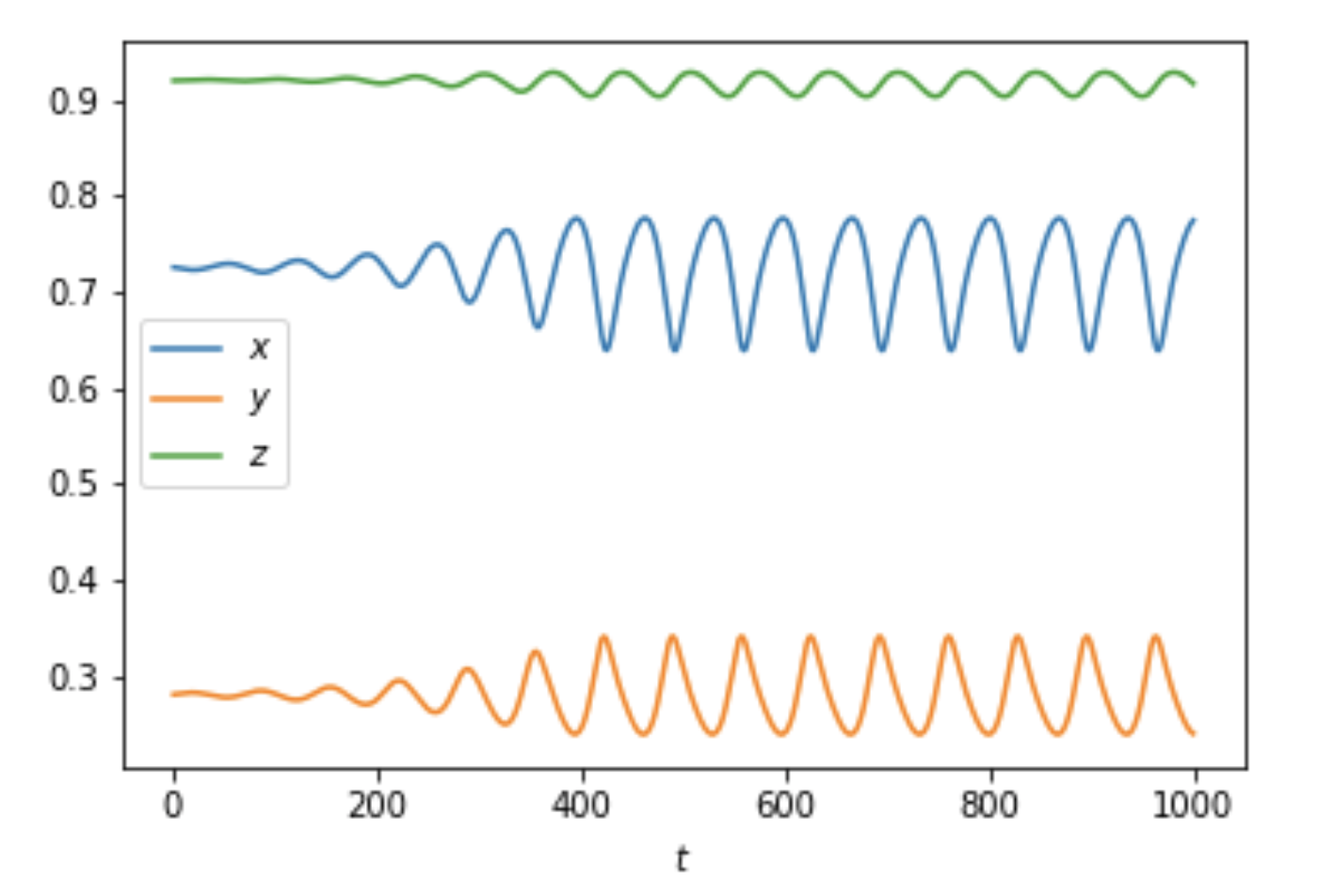}
}
\subfigure[dynamics of \eqref{foodchain}]{
\includegraphics[scale=0.25]{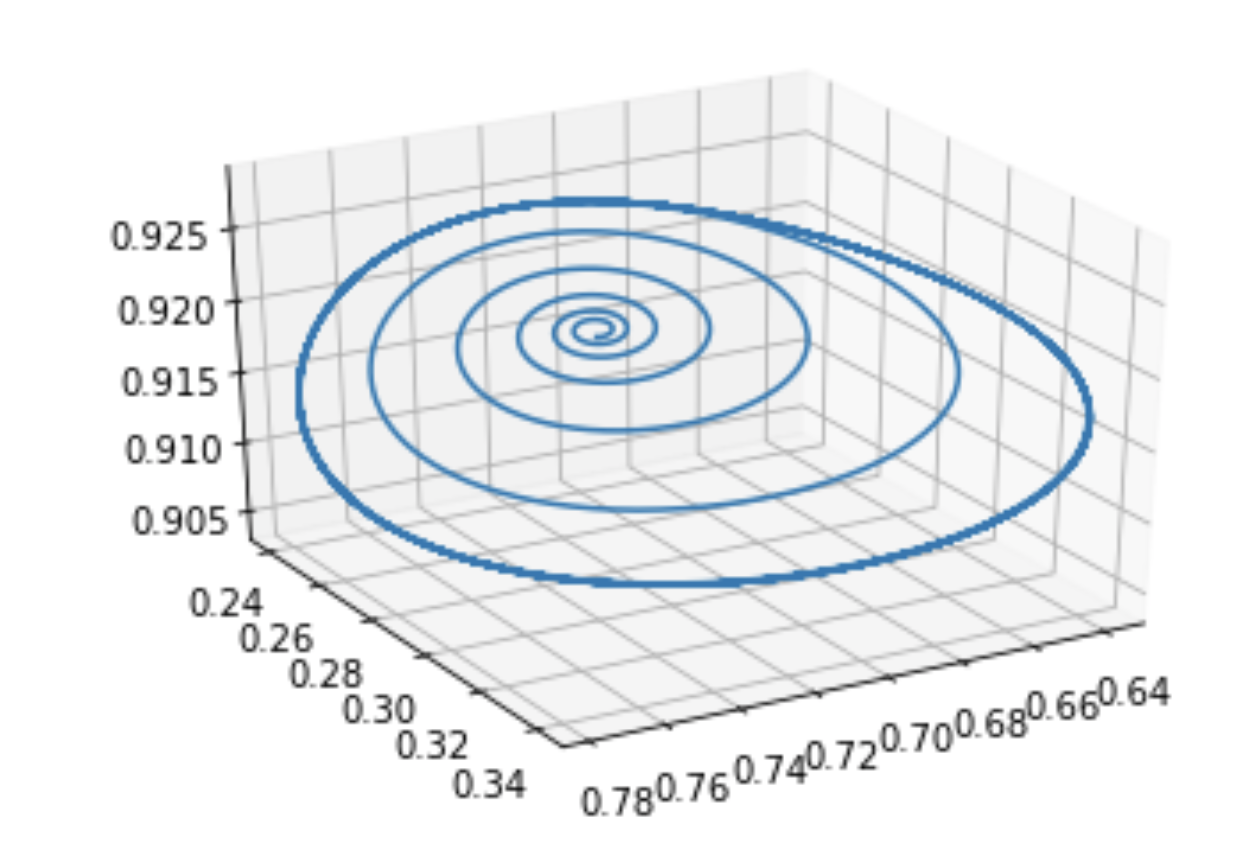}
}
\subfigure[time courses of all species]{
\includegraphics[scale=0.25]{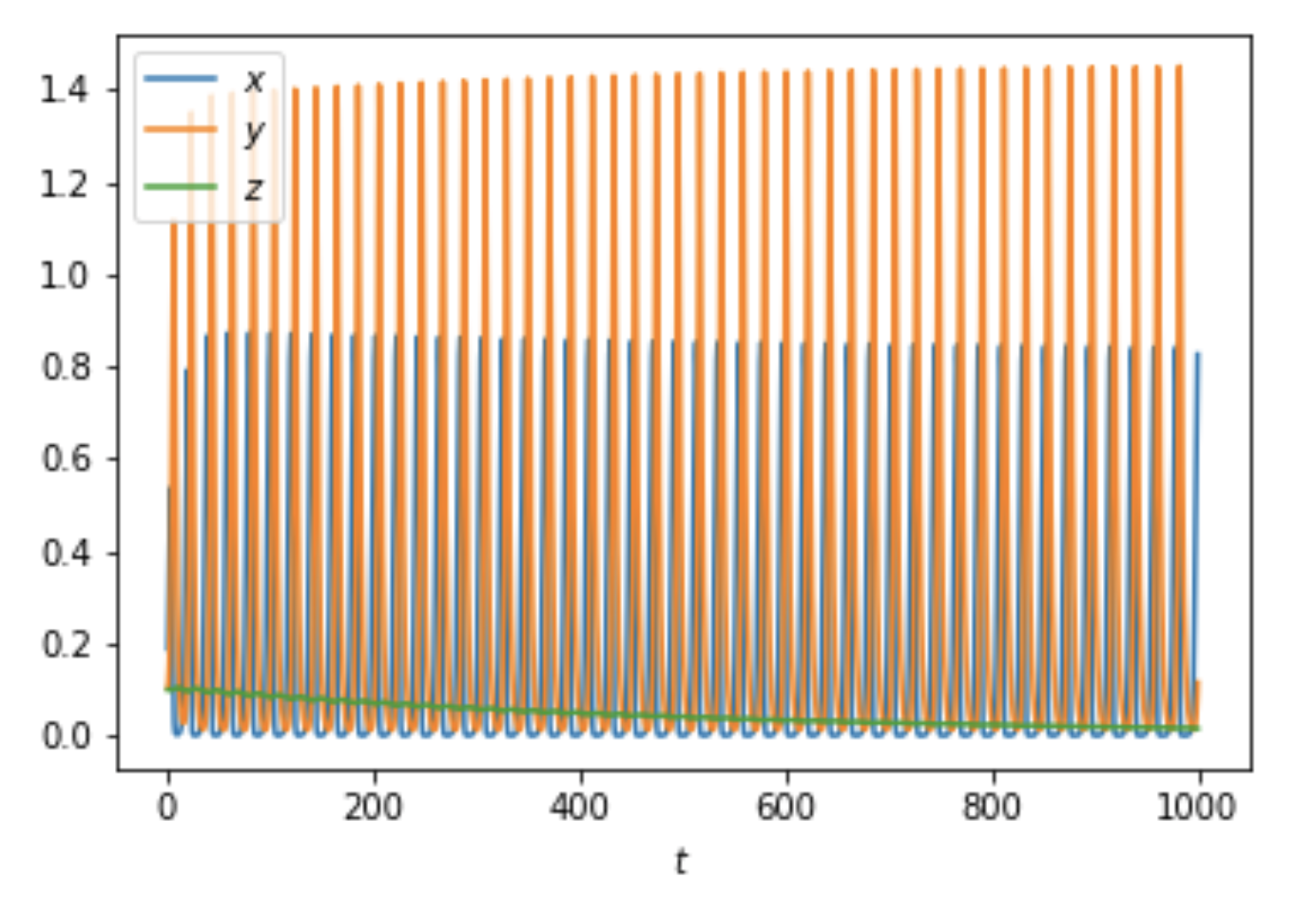}
}
\subfigure[dynamics of \eqref{foodchain}]{
\includegraphics[scale=0.25]{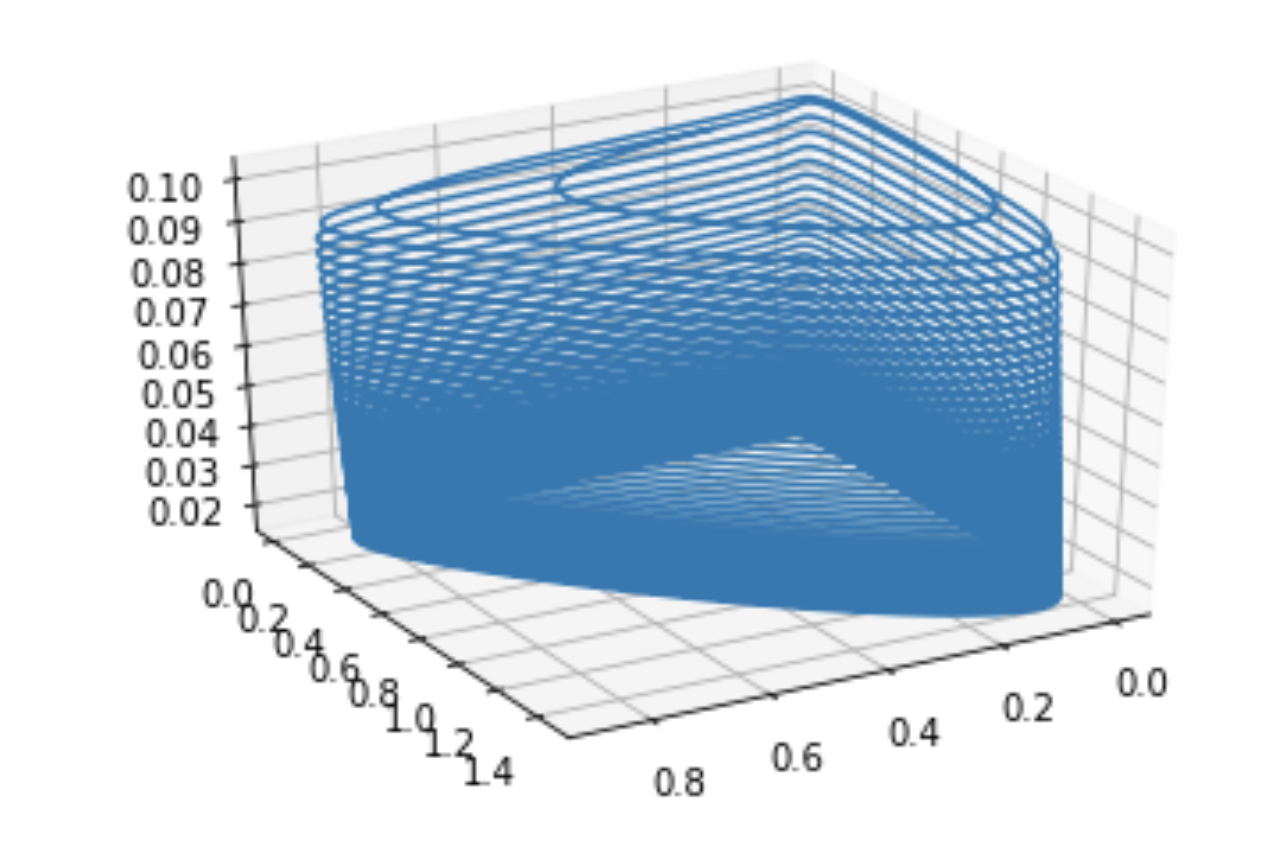}
}
\subfigure[dynamics of \eqref{foodchain}]{
\includegraphics[scale=0.3]{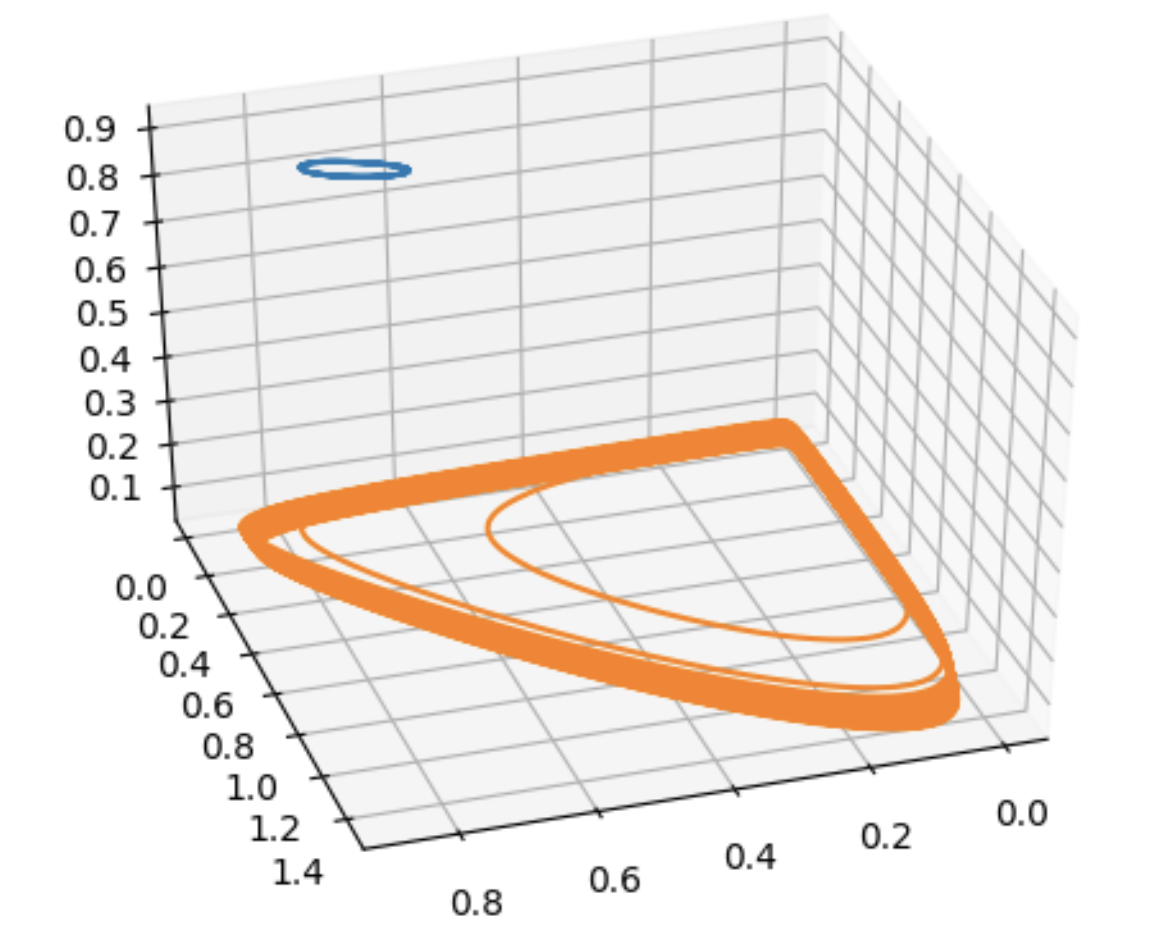}
}
\end{center}
\caption{Set parameters, $a_1 = 0.3$, $m_1 = 5/3$, $d_1 = 0.4$, $a_2 = 0.9$, $d_2 = 0.01$, and $m_2=0.042$. There is a stable interior periodic solution and a limit cycle on $H_2$. The behavior of asymptotic dynamics is dependent on the initial condition. Please see the details in the context.}\label{figure4}
\end{figure}

\item $m_2=0.065$ : For this case, there is also only one positive equilibrium which is unstable by checking the R-H criterion. Similarly, we simulate two trajectories of \eqref{foodchain} whose initial points are close to the positive equilibrium (panels (a) and (b) of Figure \ref{figure5}) and close to $\Gamma$  (panels (c) and (d) of Figure \ref{figure5}), respectively. It can be observed that the limit cycle on $H_2$ is not stable anymore by checking the time course of species $z$ (green line) in the panel (c) of Figure \ref{figure5}, which suggests that \eqref{zsurvive} holds and species $z$ survives. Moreover, in panels (a) and (b) of Figure \ref{figure5}, we can see a strange attractor similar to the results in \cite{Hastings:1991tv}, and the trajectory with initial points close to $\Gamma$ approaches the same strange attractor eventually. This numerical result suggests that the interior periodic solution is also unstable.

\begin{figure}[hp]
\begin{center}
\subfigure[time courses of all species]{
\includegraphics[scale=0.25]{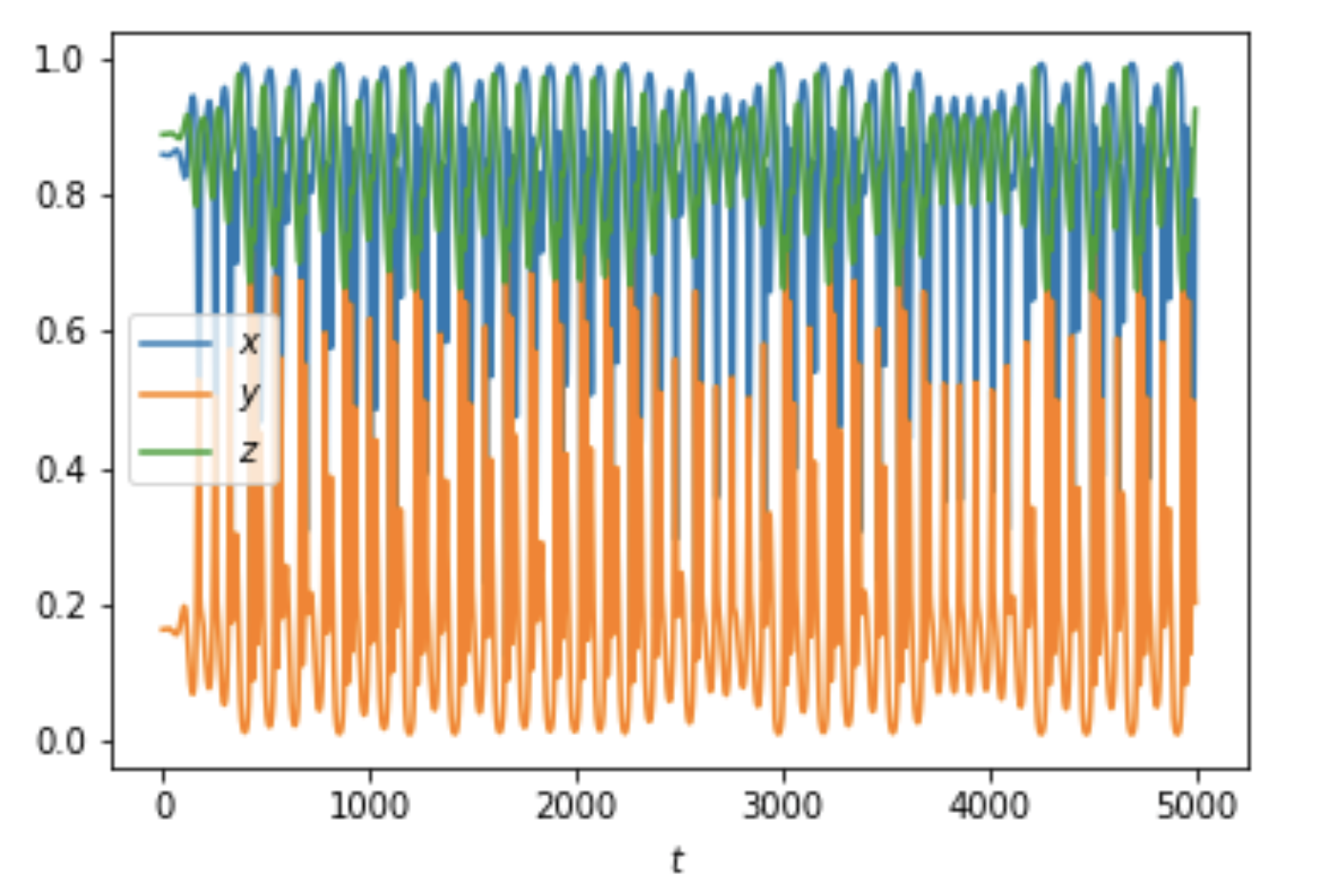}
}
\subfigure[dynamics of \eqref{foodchain}]{
\includegraphics[scale=0.25]{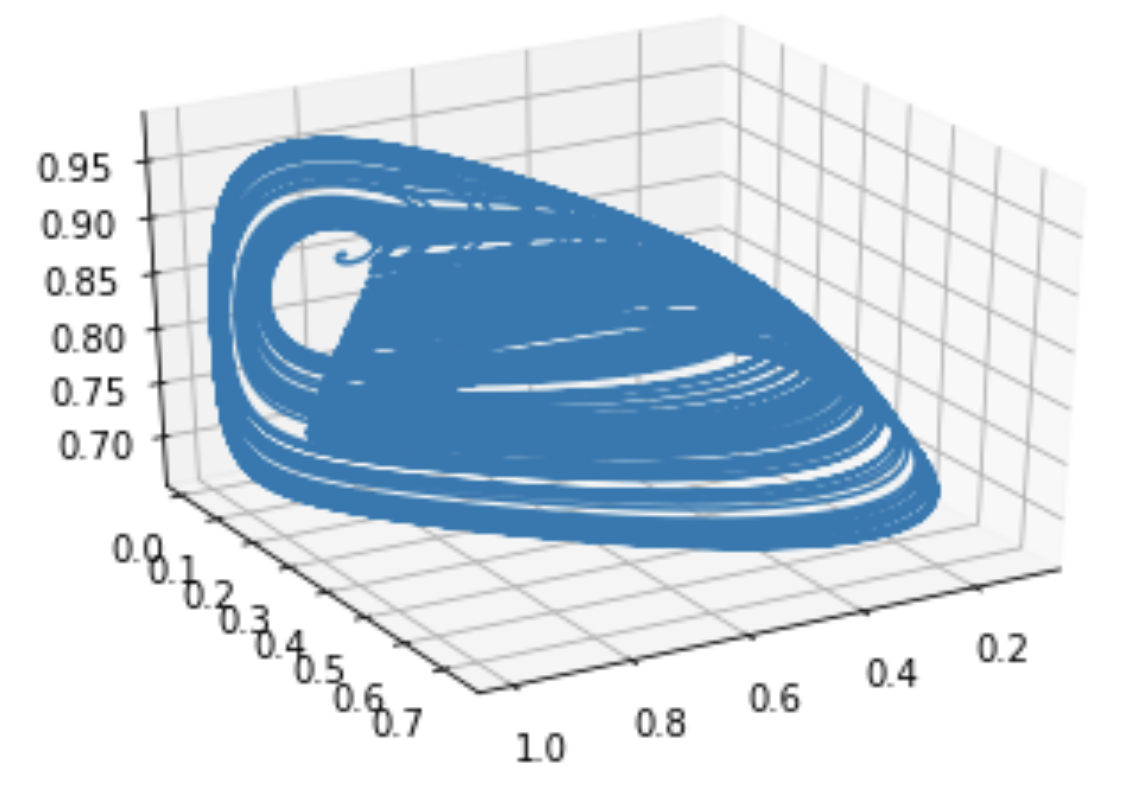}
}
\subfigure[time courses of all species]{
\includegraphics[scale=0.25]{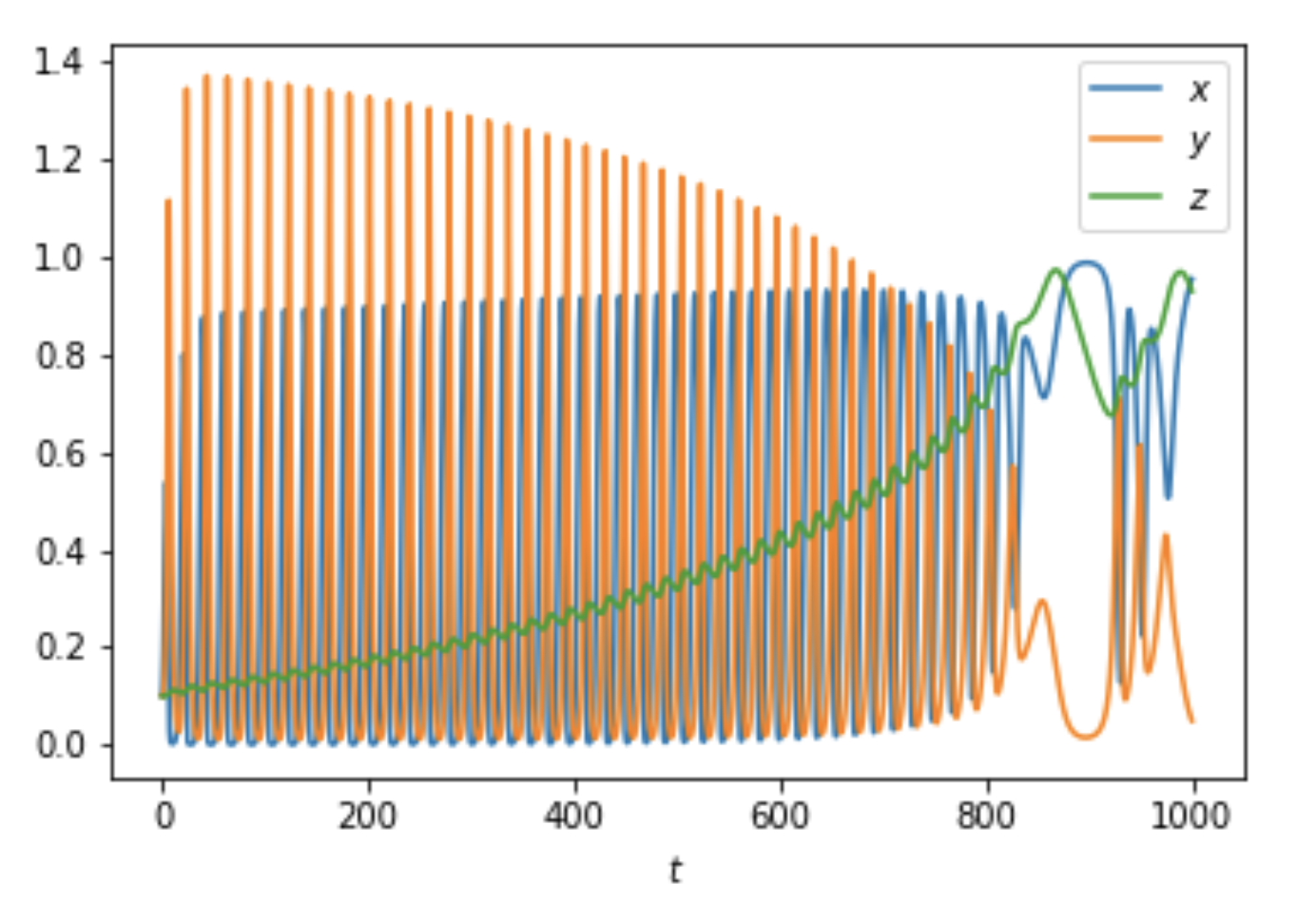}
}
\subfigure[dynamics of \eqref{foodchain}]{
\includegraphics[scale=0.25]{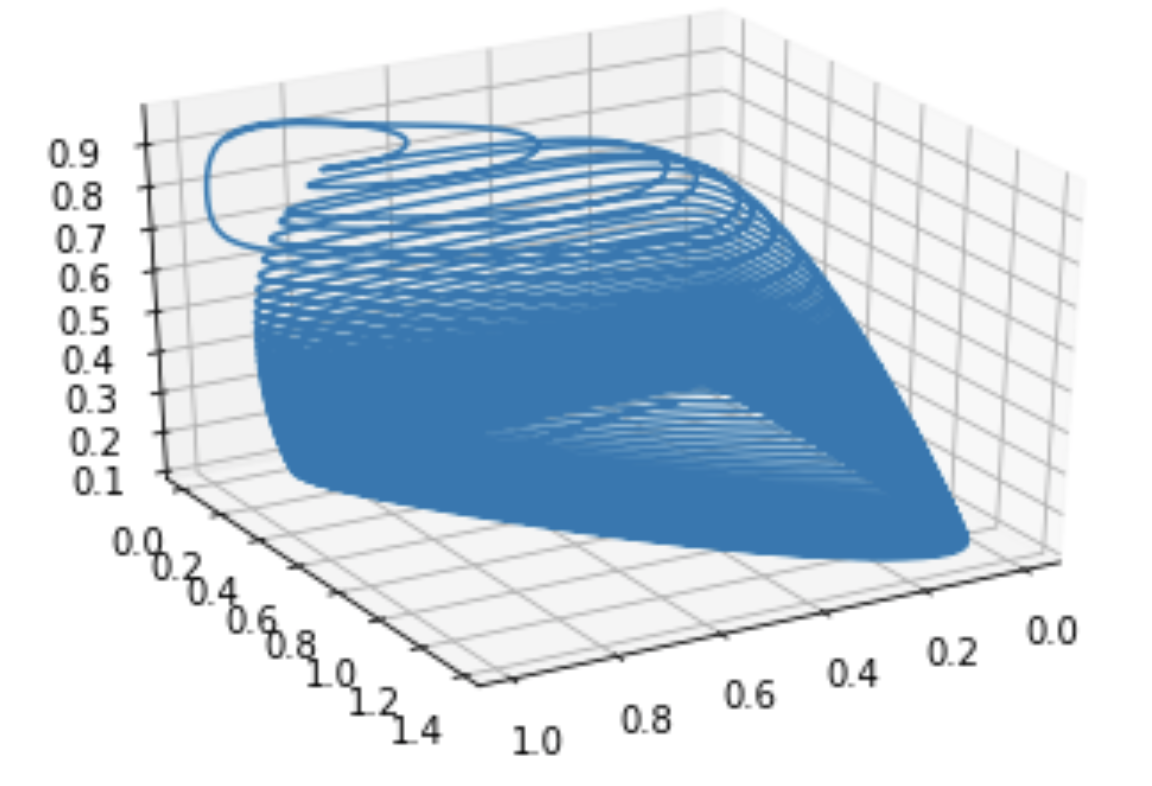}
}
\end{center}
\caption{Set parameters, $a_1 = 0.3$, $m_1 = 5/3$, $d_1 = 0.4$, $a_2 = 0.9$, $d_2 = 0.01$, and $m_2=0.065$. There is a strange attractor and a limit cycle on $H_2$ which is not stable. It is easy to see that solutions with small $z$ coordinate is attracted to the strange attractor in the panel (c). Please see the details in the context.}\label{figure5}
\end{figure}
\end{enumerate}

Finally, we would like to give some interpretations of the previous numerical simulations comparing with other similar interesting investigations \cite{Hogeweg:1978, Scheffer1991, Hastings:1991tv, Kuznetsov1996}. The values of parameters in Table 1 of \cite{Kuznetsov1996} (the parameters value with black color in Table \ref{values_table}) are transformed to fit our model \eqref{foodchain} (the parameters value with red color in Table \ref{values_table}), where the death rates $d_i$ are the same, $m_i$ of model \eqref{foodchain} equal to $a_i/b_i$, and $a_i$ is equal to $1/b_i$, respectively. Direct numerical computations show that these three cases belong to the category (II)(2)(b)(iv) of Table \ref{table1}. So we conjecture that the most likely scenario of chaos is the interplay between two unstable periodic solutions. One is $\Gamma$ and another one is bifurcated by the positive unique unstable equilibrium. 

\begin{table}[htp]
\caption{The values of parameters in Table 1 of \cite{Kuznetsov1996} (the parameters value with black color) are transformed to fit our model \eqref{foodchain} (the parameters value with red color.}\label{values_table}
\begin{center}
\begin{tabular}{l|c|c|c|c|c|c|c}
\hline
{\bf Reference} & $a_1(\textcolor{red}{m_1})$ & $a_2(\textcolor{red}{m_2})$ & $b_1(\textcolor{red}{a_1})$ & $b_2(\textcolor{red}{a_2})$ & $d_1$ & $d_2$ & Classification\\
\hline
Hogeweg et. al. \cite{Hogeweg:1978} &1.81 & 0.181 & 4.5 & 0.45 & 0.16 & 0.08 & \\
\textcolor{red}{Hogeweg et. al. \cite{Hogeweg:1978}} &\textcolor{red}{0.402} & \textcolor{red}{0.402} & \textcolor{red}{0.222} & \textcolor{red}{2.22} & \textcolor{red}{0.16} & \textcolor{red}{0.08} & \textcolor{red}{(II)(2)(b)(iv)}\\
\hline
\hline
Scheffer \cite{Scheffer1991} & 8.0 & 2.88 & 6.66 & 2.4 & 0.87 & 0.25 &\\
\textcolor{red}{Scheffer \cite{Scheffer1991}} & \textcolor{red}{1.2} & \textcolor{red}{1.2} & \textcolor{red}{0.15} & \textcolor{red}{0.417} & \textcolor{red}{0.87} & \textcolor{red}{0.25} & \textcolor{red}{(II)(2)(b)(iv)}\\
\hline
\hline
Hastings et. al. \cite{Hastings:1991tv} & 5.0 & 0.1 & 4.0 & 2.0 & 0.4 & 0.01 &\\
\textcolor{red}{Hastings et. al. \cite{Hastings:1991tv}} & \textcolor{red}{1.25} & \textcolor{red}{0.05} & \textcolor{red}{0.25} & \textcolor{red}{0.5} & \textcolor{red}{0.4} & \textcolor{red}{0.01} &\textcolor{red}{(II)(2)(b)(iv)}\\
\hline
\hline
\end{tabular}
\end{center}
\end{table}%

\end{document}